\documentclass[12pt]{amsart}

\usepackage{xspace,amssymb,amsfonts,euscript,eufrak,amsthm, amsmath}
\usepackage{graphicx} 

\usepackage{srcltx}

\RequirePackage{color}
\definecolor{myred}{rgb}{0.75,0,0}
\definecolor{mygreen}{rgb}{0,0.5,0}
\definecolor{myblue}{rgb}{0,0,0.65}

\RequirePackage{ifpdf}
\ifpdf
  \IfFileExists{pdfsync.sty}{\RequirePackage{pdfsync}}{}
  \RequirePackage[pdftex,
   colorlinks = true,
   urlcolor = myblue, 
   citecolor = mygreen, 
   linkcolor = myred, 
   pagebackref,
   bookmarksopen=true]{hyperref}
\else
  \RequirePackage[hypertex]{hyperref}
\fi

\usepackage{epstopdf}
\usepackage{pdfsync}

\usepackage{palatino}
\title[Intersection Cohomology Complexes]{Modular Intersection cohomology complexes on Flag Varieties}

\author{Geordie Williamson
\\ with an appendix by Tom Braden}

\address{ Geordie Williamson \\
Mathematical Institute,
  University of Oxford, 24-29 St Giles', Oxford, OX1 3LB, UK}
\email{geordie.williamson@maths.ox.ac.uk}
\urladdr{http://people.math.ox.ac.uk/williamsong}

\address{Tom Braden \\
Dept.\ of Mathematics and Statistics\\
         University of Massachusetts, Amherst}
\email{braden@math.umass.edu}

\input xy
\xyoption {all}

  \newcommand{\nc}{\newcommand}
  \newcommand{\renc}{\renewcommand}

\nc{\kos}[2]{\EuScript{K}_{#1,#2}}
\nc{\bet}{b}
\nc{\vp}{\varphi}
\nc{\betT}{b_T}
\nc{\QT}{Q}
\nc{\HT}{S}
\nc{\ep}{\epsilon}
\nc{\Bi}{\mathbf{i}}
\nc{\BB}{B\times B}
\nc{\C}{\mathbb{C}}
\nc{\R}{\mathbb{R}}
\nc{\Cs}[1]{\underline{\C}_{#1}}
\nc{\IC}{\mathbf{IC}}
\nc{\sGw}{G_w}
\nc{\up}{\vp^G_H}
\nc{\Kw}{K_w}
\nc{\B}{\mathcal{B}}
\nc{\SU}[1]{\mathrm{SU}(#1)}
\nc{\SL}[1]{\mathrm{SL}(#1)}
\nc{\HU}[1]{\mathbb{H}\mathrm{U}(#1)}
\DeclareMathOperator{\rk}{rk}
\renc{\t}{\mathfrak t}
\nc{\td}{\t^*}
\nc{\g}{\mathfrak g}
\nc{\HG}{H_G}
\renc{\k}{\mathbf{k}}
\renc{\S}[1]{R_{#1}}
\nc{\Si}{\S i}
\nc{\hc}{\mathbb{H}^*}
\nc{\mc}{\mathcal}
\nc{\Hom}{\mathrm{Hom}}
\nc{\ti}{\tilde}
\renc{\O}{\mathbb {O}}
\nc{\hcBB}{\hc_{B\times B}}
\nc{\si}{\sigma}
\nc{\al}{\alpha}
\nc{\HH}{\mathbb H}
\nc{\Tor}{\mathrm{Tor}}
\nc{\KR}{\mc{KR}}
\nc{\Supp}{\mathrm{Supp}}
\nc{\ASu}{\mathrm{Supp}'}
\nc{\tri}{\tau}
\nc{\ext}{\mathrm{ext}}

\nc{\K}{\mathbb{K}}

\nc{\A}{\mathcal{A}}
\nc{\AH}{\A_H}
\nc{\AG}{\A_G}
\nc{\Lotimes}{\stackrel{L}{\otimes}}
\nc{\be}{\beta}
\nc{\ga}{\gamma}
\nc{\BS}[1]{BS_{#1}}
\nc{\BSi}{\BS\Bi}
\nc{\excise}[1]{}

\nc{\ver}{\vartheta}

\nc{\Z}{\mathbb{Z}}
\nc{\N}{\mathbb{N}}
\nc{\sE}{\mathcal{E}}
\nc{\sF}{\mathcal{F}}
\nc{\sG}{\mathcal{G}}
\nc{\F}{\mathbb{F}}
\nc{\Q}{\mathbb{Q}}

\nc{\He}{\mathcal{H}}
\nc{\h}{\underline{H}}
\nc{\bo}{\underline{B}}

\nc{\dL}{\mathcal{L}}
\nc{\dR}{\mathcal{R}}

\nc{\D}{\mathbb{D}} 

\DeclareMathOperator{\sFor}{For} 
\DeclareMathOperator{\Char}{char}

\DeclareMathOperator{\ch}{ch}
\DeclareMathOperator{\End}{End}

\DeclareMathOperator{\inc}{inc}
\DeclareMathOperator{\rad}{rad}

\newcommand{\kk}{{\Bbbk}}

\newcommand{\PP}{{\mathbb P}}

\newcommand{\cN}{\mathcal{N}}
\newcommand{\cL}{\mathcal{L}}
\newcommand{\cT}{\mathcal{T}}

\newcommand{\Sig}{{\Sigma}}
\newcommand{\vep}{{\varepsilon}}
\newcommand{\udot}{{\scriptscriptstyle \bullet}}

\def\coker{\mathop{\rm coker}\nolimits}
\def\rank{\mathop{\rm rank}\nolimits}

\newcommand{\Xt}{\Sig}

\newcommand{\supp}{\mathop{\rm supp}}

\newcommand{\bw}{\mathbf{w}}

\newcommand{\wti}{\tilde{w}}
\newcommand{\sti}{\tilde{s}}

\newcommand{\defeq}{\stackrel{\rm def}{=}}
\nc{\Ad}{\mathop{\mathrm{Ad}}}

\newcommand{\bal}{{\boldsymbol{\alpha}}} 

  \newtheorem{defi}{Definition}
  \newtheorem{thm}{Theorem}[section]
  \newtheorem{lem}[thm]{Lemma}
  
  \newtheorem{prop}[thm]{Proposition}
  \newtheorem{cor}[thm]{Corollary}

  \newtheorem{ex}[thm]{Example}
  \theoremstyle{remark}
  \newtheorem{remark}[thm]{Remark}
  
\begin{document}

\begin{abstract}
  We present a combinatorial procedure (based on the $W$-graph of the
  Coxeter group) which shows that the characters of many intersection
  cohomology complexes on low rank complex flag varieties with
  coefficients in an arbitrary field are given by
  Kazhdan-Lusztig basis elements. Our procedure exploits the existence
  and uniqueness of parity sheaves. In particular we are able to show
  that the characters of
  all intersection cohomology complexes with coefficients in a field
  on the flag variety of type $A_n$ for $n < 7$ are given by
  Kazhdan-Lusztig basis elements. By results of Soergel, this implies
  a part of Lusztig's conjecture for $SL(n)$ with $n \le 7$. We also
  give examples where our techniques fail.

In the appendix by Tom Braden examples are given of intersection
cohomology complexes on the flag varities for $SL(8)$ and $SO(8)$
which have torsion in their stalks or costalks.
\end{abstract}
\maketitle

\section{Introduction}
\label{sec:introduction}

Let $\kk$ be a field of characteristic $p \ge 0$. Let $G$ be a connected
reductive algebraic group over $\C$, $B \subset G$ denote a
Borel subgroup of $G$ and let $(W, S)$ be the corresponding
Weyl group and its simple reflections. 
Consider the flag variety $G/B$ with its classical (metric) topology
and let $D^b_{\Lambda}(G/B)$ denote the bounded derived category of
sheaves of $\kk$-vector spaces on $G/B$ constructible along
$B$-orbits. In $D^b_{\Lambda}(G/B)$ there exist the intersection
cohomology sheaves $\IC(w)$. The sheaf $\IC(w)$ is supported on the
closure of the Bruhat cell $BwB/B$ and its restriction to $BwB/B$ is a
constant sheaf in degree $- \ell(w)$.

Let $\He$ be the Hecke algebra of $(W,S)$ over $\mathbb{Z}[v,v^{-1}]$ normalised so as to satisfy
\begin{eqnarray*}
H_sH_w &= \left \{ \begin{array}{ll}
H_{sw} & \text{if $sw > w$}  \\
(v - v^{-1})H_{w} + H_{sw} & \text{if $sw < w$} \end{array} \right .
\end{eqnarray*}
and let $\{ \h_w \; |\; w \in W \}$ be the Kazhdan-Lusztig basis of $\He$. It satisfies $\h_w \in H_w + \oplus_{x < w} v^{-1}\N[v^{-1}] H_x$. Given a finite dimensional graded vector space $V = \oplus V_i$ let $P(V) = \sum (\dim V_i) v^{i}$ be its Poincar\'e polynomial.

The character of a sheaf $\sF \in D^b_\Lambda(G/B)$ is the element of $\He$ given by
\begin{displaymath}
\ch(\sF) = \sum_{w \in W} P(H^*(\sF_w))v^{\ell(w)}H_w
\end{displaymath}
where $\sF_w$ denotes the stalk of $\sF$ at the point in $G/B$ corresponding to
$w \in W$. If $\kk$ is of characteristic zero, a theorem of Kazhdan and
Lusztig \cite{KL2,SpIH} says that $\ch(\IC(w)) = \h_w$.
 Thus the Poincar\'e
polynomials of the stalks of the intersection cohomology sheaves are
given by Kazhdan-Lusztig polynomials. It then follows that the same is
true in almost all characteristics\footnote{Intersection
  cohomology complexes $\IC(X,\Q)$ admit integral forms $\IC(X,\Z)$ such that the cohomology
  groups of the stalks and costalks are finitely generated (see \cite{decperv}). If the
  cohomology groups of the stalks and costalks are free of $p$-torsion (which will be the case
  for all but finitely many primes $p$), one has $\IC(X, \Z)
  \otimes_{\Z}^L \kk \cong \IC(X,\kk)$ and $\ch(\IC(X, \Q)) = \ch(\IC(X,
 \kk))$. For more details see the end of Section \ref{sec-sheaves}.}, however for any given
characteristic almost nothing is known.

It is a difficult question to determine over which fields one has
$\ch(\IC(w)) = \h_w$ and, if not, what these characters are. It has
been known since the original papers of Kazhdan and Lusztig
(\cite{KL1} and \cite{KL2}) that in non-simply laced cases the
intersection cohomology complexes may have a different character in
characteristic 2. (This happens, for example, in the only non-smooth
Schubert variety in the flag variety of $Sp(4)$.) In 2002 
Braden discovered examples  of Schubert varieties in simply
laced types $A_7$ and $D_4$ where the character of the intersection
cohomology sheaf in characteristic 2 is different to all other
characteristics (see the appendix).

In this article we define combinatorially a certain subset $\sigma(W) \subset W$ of \emph{separated elements} and show:

\begin{thm} \label{thm-sep}
Suppose that $x \in \sigma(W)$, then $\ch(\IC(w)) = \h_w$ for any field $\kk$.
\end{thm}

The determination of the characters of $\IC(w)$ is closely related to
the decomposition theorem (see \cite[6.2.5]{BBD} or
\cite[2.1.1]{dCM2}). Given a simple reflection $s \in S$ let $P_s$ be
the corresponding standard minimal
parabolic subgroup and consider the quotient map
\begin{displaymath}
G/B \stackrel{\pi_s}{\to} G/P_s.
\end{displaymath}
If $\kk$ is of characteristic zero, the decomposition theorem implies
that ${\pi_s}_* \IC(w)$ is a direct sum of shifts of intersection
cohomology sheaves. This need not be true if $\kk$ is of positive
characteristic. Given $w \in W$ and $s \in S$ let $\{ w_1, \dots, w_n
\}$ be the parameters of Kazhdan-Lusztig basis elements that
occur with non-zero coefficient when the product $\h_w\h_s$ is
expressed in the Kazhdan-Lusztig basis. Then we have:

\begin{thm} \label{thm-dec} Suppose that $w$ and $w_1, \dots, w_n$ lie
  in $\sigma(W)$. Then the decomposition theorem holds for
  ${\pi_s}_*\IC(w)$; that is, ${\pi_s}_*\IC(w)$ is isomorphic to a
  direct sum of shifts of intersection cohomology complexes.
\end{thm}

Whilst being of considerable intrinsic interest, these questions are
also important in representation theory. Assume that $\kk$ is
algebraically closed and that the characteristic of $\kk$ is
strictly greater than the Coxeter number of $W$. Now let $G^{\vee}$ be
a semi-simple
and simply connected algebraic group over $\kk$ with maximal torus
$T^\vee \subset G^\vee$ and root system dual to that of $G$ (for a
choice of maximal torus $T \subset B$). Choose a Borel subgroup
$B^\vee \supset T^\vee$ and define positive roots $R^+
\subset X(T^{\vee})$ so that the roots corresponding to $B^\vee$ are those
lying in $-R^+$. To each weight $\lambda \in X(T^{\vee})$ one may
associate a  module $H^0(\lambda)$ which is non-zero if and
only if $\lambda$ is dominant, in which case it contains a unique
simple submodule $L(\lambda)$.

A conjecture of Lusztig \cite{LusPb} expresses the characters of the
simple  $G^{\vee}$-modules $L(\lambda)$ in terms of the (known) characters of
the modules $H^0(\lambda)$. A particular case of the conjecture is the
following (see \cite{Soe}): let $\rho \in X(T^{\vee})$ denote the
half-sum of the
positive roots, and let $st = (p-1)\rho$ the Steinberg weight, then it is
conjectured that, 
\begin{equation} \label{eq:lusztig}
[H^0(st + x\rho): L(st + y\rho)] = h_{x,y}(1) \quad \text{for all $x, y \in W$,}
\end{equation}
where $h_{x,y} \in v^{-1}\N[v^{-1}]$ is the Kazhdan-Lusztig polynomial indexed by $x, y \in W$. A theorem of Soergel \cite{Soe} says that (\ref{eq:lusztig}) is equivalent to the semi-simplicity of ${\pi_s}_*\IC(x)$ for all $x \in W$ and $s \in S$.

Of course, in order to apply Theorems \ref{thm-sep} and \ref{thm-dec}
it is necessary to know the set $\sigma(W)$. The essential ingredient
in the calculation of $\sigma(W)$ is the $W$-graph of the Coxeter
system $(W,S)$. Unfortunately, even in simple situations the $W$-graph
can be very complicated and no general description is known. However,
using Fokko du Cloux's program \emph{Coxeter} \cite{dCl1} it is
possible to use a computer to determine the set $\sigma(W)$ for low
rank Weyl groups. The simplest situation is when $\sigma(W) = W$.
This only occurs in type A in low rank: 

\begin{thm} \label{thm:a6}
Let $G$ be of type $A_n$ for $n \le 6$. Then $\sigma(W) = W$. Hence,
in all characteristics the intersection cohomology complexes have
characters given by Kazhdan-Lusztig basis elements and the
decomposition theorem holds for ${\pi_s}_* \IC(x)$ for all $s \in S$
and $x \in W$.
\end{thm}

It also follows that (\ref{eq:lusztig}) holds for $G^{\vee} = SL_n(k)$
if $n \le 7$ and $\kk$ has characteristic $> n+1$.


In other types and type $A_n$ for $n \ge 7$ our techniques are not as
effective. In most examples that we have computed $\sigma(W)$ is not
the entire Weyl group. However, we are able to show that the
characters are given by Kazhdan-Lusztig basis elements and verify 
an analogue of the decomposition theorem for many $w \in W$
(that is those $w \in \sigma(W)$)
in ranks $\le 6$. It also seems that the elements $x \notin \sigma(W)$
(for which our methods fail) will provide an interesting source of
future research.

Indeed in the appendix Braden shows that, both in type $D_4$ and
$A_7$, the intersection cohomology sheaf over the integers corresponding to a 
 minimal element in $W \setminus \sigma(W)$ has
2-torsion in the cohomology groups of its stalks or costalks (and hence $\ch(\IC(w)) \ne \h_w$ if the
coefficients are taken to be of characteristic 2). These two
examples, together with the case of dihedral groups, leads one to
suspect a close relationship between $W \setminus \sigma(W)$, and
those intersection cohomology complexes which have 
torsion in the cohomology groups of their stalks or costalks over $\Z$. It would be interesting to
have more examples in this direction.

Let us briefly mention that, in \cite{BrM} Braden and MacPherson give an
algorithm for the calculation of the stalks of the intersection
cohomology complexes with coefficients in $\mathbb{Q}$, using only data that can be obtained from the
fixed points and one-dimensional orbits of a maximal torus acting on
the flag variety. (This data is encoded
in the so-called ``moment graph'' of the flag variety.) 
The recent paper \cite{FW} of Fiebig
and the author extends this result, showing that the moment graph of the flag variety can be
used to calculate the characters of parity sheaves (a certain class of sheaves
characterised by the vanishing of stalks and costalks in degrees of a
fixed parity). It follows that this algorithm can be used to 
determine those intersection cohomology complexes which have torsion
in the cohomology groups of their stalks or costalks over the integers (see Corollary
\ref{cor:Ztorsion}). Thus the results of this paper
could (at least in principle) be deduced from the moment graph. 
In fact, the computations of torsion
in the appendix translate easily into the moment
graph language and give a proof that the moment
graph sheaves obtained via direct image from certain Bott-Samelson
resolutions do not split as much as expected
unless $2$ is invertible in the coefficient ring.

The structure of the paper is as follows. In Section \ref{sec-hecke}
we review the Hecke algebra and Kazhdan-Lusztig basis in more detail
and recall the $W$-graph associated to $(W,S)$. In Section
\ref{sec-sheaves} we discuss \emph{parity sheaves}, which are our main
theoretical tool. In Section
\ref{sec-sep} we define the subset $\sigma(W) \subset W$ and prove
Theorems \ref{thm-sep} and \ref{thm-dec}. In Section \ref{sec-comp} we
discuss the calculation of the sets $\sigma(W)$ via computer and give
some examples of the sets $\sigma(W)$ for low rank Weyl groups.

\subsection{Acknowledgments} I would like to thank Tom Braden for
useful correspondence, pointing out errors in previous versions, and
contributing the appendix. Both G.W. and T.B. would like to extend
their gratitude to Olaf Schn\"urer for very detailed feedback on a
previous version of this manuscript, which lead to the rewriting of
Section \ref{sec-sep}, as well as smaller improvements on almost every
page! We would also like to thank Simon Riche and Patrick Polo for pointing out some typos.

\section{The Hecke algebra and $W$-graphs} \label{sec-hecke}

In this section we recall the Hecke algebra and Kazhdan-Lusztig basis
in slightly more detail. Up to some small changes of notation we
follow \cite{LuBook}. Let $(W,S)$ be a Coxeter system with Bruhat
order $\le$ and length function $\ell: W \to \mathbb{N}$. Given $w \in
W$ we define the \emph{left} and \emph{right descent set} to be
\begin{displaymath}
\dL(w) = \{ s \in S \; | \; sw < w \} \; \text{  and  } \;
\dR(w) = \{ s \in S \; | \; ws < w \}.
\end{displaymath}
Recall that the Hecke algebra is the free $\Z[v,v^{-1}]$-module with multiplication given by
\begin{eqnarray*}
H_sH_w &= \left \{ \begin{array}{ll}
H_{sw} & \text{if $s \notin \dL(w)$,}  \\
(v - v^{-1})H_{w} + H_{sw} & \text{if $s \in \dL(w)$.} \end{array} \right .
\end{eqnarray*}
The elements $H_w$ are invertible and there is an involution $h
\mapsto \overline{h}$ on $\He$ which sends $H_w$ to $H_{w^-1}^{-1}$ and $v$ to $v^{-1}$. We will call elements fixed by this involution \emph{self-dual}.

There exists a basis $\{ \h_w \}$ of $\He$ called the \emph{Kazhdan-Lusztig basis} which is uniquely determined by requiring:
\begin{enumerate}
\item the $\h_w$ are self-dual;
\item $\h_w = \sum_{x \le w}h_{x,w}H_x$ where $h_{w,w} = 1$ and $h_{x,w} \in v^{-1}\Z[v^{-1}]$ for $x \ne w$.
\end{enumerate}
The polynomials $h_{x,w}$ are (up to a renormalisation) the \emph{Kazhdan-Lusztig polynomials}. One may check, for example, that $\h_s = H_s + v^{-1}H_{id}$.

The action of $\h_s$ for $s \in S$ on the Kazhdan-Lusztig basis has a
particularly simple form. We denote by $\mu(x,w)$ the coefficient of
$v^{-1}$ in $h_{x,w}$. Then (see \cite[Theorem 6.6 and Corollary 6.7]{LuBook}):
\begin{eqnarray} \label{KL-mult-left}
\h_s\h_w &= \left \{ \begin{array}{ll}
(v + v^{-1})\h_{w} & \text{if $s \in \dL(w)$,}  \\
\h_{sw} + \sum_{x<w; s \in \dL(x)} \mu(x,w)\h_x & \text{if $s \notin \dL(w)$.} \end{array} \right .
\end{eqnarray}
Simlarly, on the right we have:
\begin{eqnarray} \label{KL-mult-right}
\h_w\h_s &= \left \{ \begin{array}{ll}
(v + v^{-1})\h_{w} & \text{if $s \in \dR(w)$,}  \\
\h_{ws} + \sum_{x<w; s \in \dR(x)} \mu(x,w)\h_x & \text{if $s \notin \dR(w)$.} \end{array} \right .
\end{eqnarray}
It is known if $W$ is a Weyl group (the case of
interest below) then $\mu(x,w) \in \mathbb{N}$ for all $x, w \in W$.

Thus all the information about the action of $\h_s$ on the left and
right on the Kazhdan-Lusztig basis may be encoded in a labelled graph,
known as the \emph{$W$-graph}. The vertices correspond to the elements
of $W$ and are labelled with the left and right descent sets. There is
a directed edge between $x$ and $y \in W$ if $\mu(x,y) \ne 0$, in
which case the edge is labelled with the value of
$\mu(x,y)$. (Strictly speaking, the graph that we define is a
  variant of what Kazhdan and Lusztig call a $W \times
  W^{\textrm{o}}$-graph.) For more
details on the Kazhdan-Lusztig basis and $W$-graphs the reader is
referred to \cite{KL1}, \cite{HUM}, \cite{LuBook} or \cite{SoeKL}.

\section{Parity sheaves} \label{sec-sheaves}

\nc{\triright}{\stackrel{[1]}{\to}}
\nc{\uk}{\underline{k}}
\nc{\sur}{\twoheadrightarrow}

In this section we recall some basic properties of ``parity sheaves'' introduced in \cite{JMW2} and motivated by \cite{Soe}. These are our main technical tool.

We recall briefly the setting of \cite{JMW2}. Throughout, $\kk$ denotes a field or
complete local principal ideal domain.\footnote{The case where $\kk$ is a
  complete local PID will only be necessary at the end of this
  section to prove Theorem \ref{thm:Ncomb} and discuss the
  relationship between the characters of parity sheaves and the
  existence of torsion in the cohomology groups of the stalks or costalks of intersection
  cohomology complexes. If one is willing to
  accept these results one may assume that $\kk$ is field throughout.}
All spaces will be complex algebraic $H$-varieties,
for $H$ a complex linear algebraic group.
Given an $H$-space $X$, we write $D^b_c(X,\kk)$ or $D^b_c(X)$ for the bounded derived
category of constructible $\kk$-sheaves on $X$ and $D^b_H(X,\kk)$ or $D^b_H(X)$ for the
bounded $H$-equivariant derived category of constructible
sheaves of $\kk$-modules on $X$ (see \cite{BLu}).
Given $\sF$ in $D^b_c(X)$  or $D^b_{H}(X)$ we denote by
$\mathcal{H}^j(\sF)$ the $j^{th}$ cohomology sheaf of $\sF$ (which is
a sheaf or equivariant sheaf of $\kk$-modules).
By abuse of language, we
call objects in $D^b_{H}(X)$ sheaves. 
We denote by $\sFor : D^b_{H}(X) \to
D^b_{c}(X)$ the forgetful functor (see \cite{BLu}).
If $H$ has finitely many orbits on $X$ then the image of the forgetful
functor is contained in $D^b_{\Lambda}(X)$, the full subcategory of
$D^b_c(X)$ consisting of sheaves whose cohomology sheaves are locally constant
along $H$-orbits. The category  $D^b_H(X)$ is Krull-Remak-Schmidt:
an object is indecomposable if and only if its endomorphism ring is local; any object
admits a decomposition into indecomposable objects; and the
multiplicity of an indecomposable object as a summand of any object is
independent of the chosen decomposition.

All maps will be equivariant morphisms of complex algebraic
varieties.
Given a map $f : X \to Y$ we have functors $f_*$, $f_!$
from $D^b_{H}(X)$ to $D^b_{H}(Y)$ and $f^*$, $f^!$
from $D^b_{H}(Y)$ to
$D^b_{H}(X)$. Similar functors exist between $D^b_c(X)$ and $D^b_c(Y)$.
On the categories
$D^b_H(X), D^b_H(Y), D^b_c(X)$ and $D^b_c(Y)$ we have the Verdier duality functor, which
we denote by $\mathbb{D}$. We have isomorphisms of
functors $\mathbb{D} f_* \cong f_! \mathbb{D}$ and $\mathbb{D} f^*
\cong f^! \mathbb{D}$. All functors $f^*, f^!, f_!, f_*, \mathbb{D}$ commute with the forgetful functor.


Now let $G$ denote a connected reductive complex algebraic group and $B \supset T$ a
 Borel subgroup and maximal torus. Let $W$ denote the Weyl group and
 $S \subset W$ the set of simple reflections corresponding to $B$. Throughout $X = G/P$,
 where $P$ is either $B$ or a minimal
 standard parabolic subgroup $P_s$ corresponding to $s \in S$
 (i.e. $P_s := \overline{BsB}$).
We regard $X$ as a $B$-variety. Each
$B$-orbit is isomorphic to an affine space and the strata are
classified by $W$ if $P = B$ and $W/\langle s \rangle$ if $P = P_s$.
Given $w \in W$ (resp. $\overline{w} \in W/\langle s \rangle$)
we denote by $X_w$ (resp. $X_{\overline{w}}$) the stratum
$BwB/B$ (resp. $BwP_s/P_s$), by $i_w : X_w \hookrightarrow G/B$
(resp. $i_{\overline w} : X_{\overline w} \hookrightarrow G/P_s$) its
inclusion and by $\underline{\kk}_w$ (resp. $\underline{\kk}_{\overline w}$) the
$B$-equivariant constant sheaf on $X_w$ (resp. $X_{\overline w}$) with
fibre $\kk$.



For brevity, $? \in \{ *, ! \}$ and $X = G/B$. A sheaf $\sF \in D^b_{B}(X)$ is
\emph{$?$-even} if $i_{w}^? \sF$ is isomorphic to 
a direct sum of even shifts of constant sheaves $\underline{\kk}_w$, for all strata $X_w \subset X$. A sheaf
is \emph{even} if it is both $*$- and $!$-even. A
sheaf $\sF$ is \emph{($?$-) odd} if $\sF[1]$ is ($?$-) even. A sheaf
$\sF \in D^b_{B}(X)$ is (?-)\emph{parity} if we have an isomorphism $\sF
\cong \sF_0 \oplus \sF_1$ with $\sF_0$ (?-)even and $\sF_1$
(?-)odd. Note that direct sums and summands of (?-)parity sheaves are
(?-)parity. Entirely analagous definitions apply when $X = G/P_s$.

The following theorem shows that one may classify indecomposable parity sheaves on the flag variety in a similar way to intersection cohomology sheaves:

\begin{thm}[{\cite[2.9]{JMW2}}] \label{thm-clas}  For any $w \in W$
  (resp. $w \in W/\langle s \rangle$) 
there exists (up to isomorphism) a unique indecomposable parity sheaf
$\sE(w, k) \in D^b_{B}(G/B)$ (resp. $\sE(\overline{w}, k) \in
D^b_{B}(G/P_s)$) with support contained in $\overline{X_w}$
(resp. $\overline{X_{\overline{w}}}$) 
and $i_w^* \sE(w) \cong \underline{\kk}_w[ \dim X_w]$ (resp. $i_{\overline{w}}^*
\sE(\overline{w}) = \underline{\kk}_{\overline{w}}[ \dim X_{\overline{w}}]$). Each $\sE(w,\kk)$
(resp. $\sE(\overline{w},\kk)$) is self-dual and any indecomposable parity
sheaf on $G/B$ (resp. $G/P_s)$ is isomorphic to
$\sE(w,\kk)[m]$ for some $w \in W$ and $m \in \mathbb{Z}$
(resp. $\sE(\overline{w},\kk)[m]$ for some $w \in W/ \langle s \rangle$
and $m \in \mathbb{Z}$). \end{thm}

If the context is clear we will write $\sE(w)$ instead of $\sE(w,\kk)$
 and $\sE(\overline{w})$ instead of $\sE(\overline{w},\kk)$.


Given a $*$-parity sheaf $\sE \in D^b_B(G/B)$ and $w \in W$ we
may write $i_w^* \sE \cong V(w) \otimes_{\kk} \underline{\kk}_w$ for some finitely-generated graded free
$\kk$-module $V(w) = \oplus V(w)_i$. We define the character of
$\sE$ in the Hecke algebra to be
\begin{equation*}
\ch(\sE) = \sum_{i \in \mathbb{Z}, w \in W} (\rk  V(w)_i)v^{\ell(w) + i}H_w
\end{equation*}
where $\rk V(w)_i$ denotes the rank of the free $\kk$-module $V(w)_i$.

\begin{remark} If $\kk$ is a field and $\sE$ is a $*$-parity sheaf then it is easily seen that
  $\ch(\sE)$ agrees with the character of $\sFor (\sE) \in
  D^b_{\Lambda}(G/B)$ as defined in the introduction.
\end{remark}

A similar character map exists for $D^b_B(G/P_s)$ for a simple
reflection $s \in S$. Let $\He_s$ denote
the left ideal $\He \h_s$ in $\He$. Then $\He_s$ is free with basis
$H_w \h_s$ for $w \in W^s$, where $W^s \subset W$ denotes the subset
of elements $w \in W$ such that $s \notin \dR(w)$. 
Given a $*$-parity sheaf $\sE \in
D^b_{B}(G/P_s)$ and $w \in W^s$ we can write $i_{\overline{w}}^* \sE \cong
V(w) \otimes \underline{\kk}_{\overline{w}}$ for some finitely-generated graded free 
$\kk$-module $V(w)= \oplus V(w)_i$. We define
the character of $\sE$ to be:
\begin{equation*}
\ch(\sE) = \sum_{i \in \mathbb{Z}, w \in W^s} (\rk V(w)_i)v^{\ell(w) + i}H_w\h_s.
\end{equation*}

For any $s \in S$ we have obvious maps given by inclusion and multiplication:
\begin{equation*}
\xymatrix@C=3cm{ \He \ar@(ur, ul)[r]^{ \cdot \h_s} & \ar@(dl, dr)[l]^{\inc} \He_s}
\end{equation*}
The quotient map $\pi_s : G/B \to G/P_s$ induces functors:
\begin{equation*}
\xymatrix@C=1cm{ D^b_B(G/B) \ar@(ur, ul)[r]^{ \pi_{s*}} &
  \ar@(dl, dr)[l]^{\pi_s^*} D^b_{B}(G/{P_s})}
\end{equation*}
The following lemma is well-known (see \cite{SpIH}, Lemme 2.6):

\begin{lem} \label{lem-parchar}
\begin{enumerate}
\item If $\sE \in D^b_B(G/B)$ is parity, then so is ${\pi_s}_* \sE$ and
\begin{equation*}
\ch({\pi_s}_* \sE) = \ch(\sE) \h_s.
\end{equation*}
\item If $\sE \in D^b_{B}(G/P_s)$ is parity, then so is $\pi_s^* \sE$ and
\begin{equation*}
\ch(\pi_s^* \sE[1]) = \inc (\ch(\sE)).
\end{equation*}
\item If $\sE \in D^b_B(G/B)$ or $D^b_{B}(G/P_s)$  is parity then
  \[ \ch(\sE[1]) = v^{-1} \ch(\sE). \]
\end{enumerate}
\end{lem}

\begin{proof}
We first show the first three relations for $*$-parity sheaves. Statement (3) is a
straightforward consequence of the definitions and (2) follows from
the definitions and the fact that $\pi_s^{-1}(X_{\overline{w}}) = X_w
\sqcup X_{ws}$ for $w \in W$. It remains to show (1).

Let $\sE$ be a $*$-parity sheaf. 
We prove (1) by induction on the number of $w \in W$ for which $i_w^*
\sE \ne 0$.
 If this number is one, then (by definition of $*$-parity) $\sE$ is necessarily
isomorphic to a direct sum of shifts of $i_{w!} \underline{\kk}_w$, for some $w \in
W$. We may assume that $\sE \cong i_{w!} \underline{\kk}_w[\ell(w)]$.
Let us write $\overline{w}$ for the image of $w$ in $W/\langle s
\rangle$. If $ws > w$ then $\pi_s$ restricts to an isomorphism $X_w
\to X_{\overline w}$. Hence 
\[ \pi_{s*} \sE \cong i_{\overline{w}!}
\underline{\kk}_{\overline w}[\ell(w)]. \]
If $ws < w$ then the restriction of $\pi_s$ to $X_w$ induces a (trivial)
$\mathbb{C}$-bundle over $X_{\overline{w}}$, hence
\[ \pi_{s*} \sE \cong i_{\overline{w}!}
\underline{\kk}_{\overline w}[\ell(w)-2]. \]
A simple calculation in the Hecke algebra then shows that in both cases
\[
\ch(\pi_{s*} \sE ) = \ch( \sE)\h_s
\]
as claimed.

We now turn to the general case. We may assume without loss of
generality that $\sE$ is $*$-parity. Choose $w \in W$ so that $X_w$ is
open in the support of $\sE$ and let $i : \overline{\supp \sE}
\setminus X_w \hookrightarrow G/B$ denote the inclusion. Then $i_w^!
\sE \cong i_w^* \sE$ and we have a distinguished triangle of $*$-parity
sheaves
\[
i_{w_!}i_w^* \sE \to \sE \to i_*i^* \sE \stackrel{[1]}{\to}
\]
By induction $\pi_{s*}$ applied to the first or third term is
$*$-parity and (1) holds. It follows that the same
is true of $\sE$ because $\ch(\sE) = \ch(i_{w_!}i_w^* \sE ) + \ch(i_*
i^* \sE)$ and
$\ch(\pi_{s*} \sE) = \ch(\pi_{s*} i_{w_!}i_w^* \sE ) + \ch(\pi_{s*}
i_* i^* \sE)$.

It remains to see that $\pi_{s*}$ and $\pi_s^*$ preserve the classes of
parity sheaves. However this follows immediately because $\mathbb{D}$
interchanges $*$-parity and $!$-parity sheaves, $\pi_{s*}
\mathbb{D} \cong \mathbb{D} \pi_{s*}$ (as $\pi_s$ is proper) and
$(\pi_s^*[1])  \mathbb{D}  \cong  \mathbb{D} ( \pi_s^*[1])$ (because $\pi_s$
is a smooth fibration with fibres of complex dimension 1).
\end{proof}

Consider $G$ as a $B \times B$-space via $(b_1, b_2) \cdot g := b_1 g
b_2^{-1}$.
 As the second copy of $B$-acts freely on $G$, the
quotient equivalence (\cite[2.6.2]{BLu}) yields an equivalence of triangulated
categories:
\[
Q^*: D^{b}_B(G/B) \stackrel{\sim}{\to} D^b_{B \times B}(G).
\]
Consider the inversion map $i : G \to G$. Then this is $B \times
B$-equivariant with respect to the swap map $B \times B \to B \times
B : (b_1, b_2) \mapsto (b_2, b_1)$. This induces an equivalence
\[
i^* : D^b_{B \times B}(G) \to D^b_{B \times B}(G)
\]
Consider the functor $\iota := (Q^*)^{-1}i^* Q^{*} : D^b_B(G/B) \to
D^b_B(G/B)$. Then $\iota$ commutes with $\mathbb{D}$ (see \cite[7.5.2]{BLu}).
It is easy to see that $\iota$ preserves parity sheaves
and that, for a parity sheaf $\sE \in D^b_B(G/B)$,
\begin{equation} \label{eq:jinvolution}
\ch ( \iota (\sE) ) = j( \ch ( \sE))
\end{equation}
where $j : \He \to \He$ is the anti-involution defined by $j(H_w)=
H_{w^{-1}}$ and $j(v) = v$.

Define endofunctors on $D^b_{B}(G/B)$ by
\[
(-)\ver_s := \pi_s^* {\pi_s}_*(-)[1] \text{ and }  \ver_s (-) := 
\iota \pi_s^* {\pi_s}_* \iota(-)[1].
\]
Then the functors $(-)\ver_s$ and $\ver_s (-)$ preserve parity
sheaves; the shift is chosen so that $\D( \ver_s \sF) \cong \ver_s (\D\sF)$ and
$\D( \sF \ver_s) \cong (\D \sF) \ver _s$. By \eqref{eq:jinvolution}
and the above lemma,
\begin{equation*}
\ch(\sE \ver_s) = \ch(\sE)\h_s \text{ and }  \ch(\ver_s \sE) =  \h_s\ch(\sE)
\end{equation*}
for parity sheaves $\sE \in D^b_B(G/B)$.

The first result about the characters of parity sheaves is the following:

\begin{prop} \label{cor-dual} For all $w \in W$, 
$\ch(\sE(w)) \in \He$ is self-dual. \end{prop}

\begin{proof} We proceed via induction on $\ell(w)$ with the base case
  being trivial. Let us fix $w$ and choose $s \in S$ with $ws < w$. By
  Theorem \ref{thm-clas} we may write 
\begin{equation*}
(\sE(ws)) \ver_s  \cong \sE(w) \oplus \sG
\end{equation*}
where
\begin{equation*}
\sG \cong \bigoplus_{x < w \atop \eta \in \Z } \sE(x)[\eta] ^{\oplus m_{x,\eta}}.
\end{equation*}
The Verdier self-duality of $(\sE({ws}))\ver_s $, $\sE({w})$ and each $\sE(x)$ for
$x < w$ together with Krull-Remak-Schmidt implies
\begin{equation*}
m_{x, -\eta} = m_{x, \eta}.
\end{equation*}
By induction, the $\ch(\sE(x))$ for $x < w$ are self-dual. Hence both $\ch(\sG)$ and $\ch(\ver_s \sE({sw})) = \h_s \ch(\sE({sw}))$ are self-dual. Thus so is $\ch(\sE(w))$.
\end{proof}



Let $s \in S$ be a simple reflection. The next proposition relates parity sheaves on $G/P_s$ to those on $G/B$:

\begin{prop} \label{cor:Ps} Let $w \in W$ be such that $ws < w$ and
  denote by $\overline{w}$ the image of $w$ in $W/\langle s \rangle$.
We have isomorphisms
\begin{equation*}
\pi_s^* \sE({\overline{w}})[1] \cong \sE(w)
\end{equation*}
and
\begin{equation*}
{\pi_s}_* \sE(w) \cong \sE({\overline{w}})[-1] \oplus \sE({\overline{w}})[1].
\end{equation*}
\end{prop}

\begin{proof} As $\sE(w)$ is a direct summand of
  $\pi_s^*\sE({\overline{w}})[1]$ and the restriction of 
$\pi_s^*\sE({\overline{w}})[1]$ to $U = X_w \sqcup X_{ws}$ is isomorphic
to a shifted constant sheaf $\underline{\kk}_U[\ell(w)]$ (and hence is indecomposable) we have
\begin{equation*}
\ch(\sE(w)) = H_w + v^{-1}H_{ws} + \sum_{x < w \atop x \ne ws} m_x H_x
\end{equation*}
for some $m_x \in \mathbb{N}[v,v^{-1}]$.
It follows (by considering $i_{\overline{w}}^* {\pi_s}_* \sE(w)$) that
\begin{equation}
{\pi_s}_* \sE(w) \cong \sE({\overline{w}})[1] \oplus \sE({\overline{w}})[-1] \oplus \sG
\end{equation}
for some parity sheaf $\sG$. We may also decompose
\begin{equation*}
\pi_s^* \sE({\overline{w}})[1] \cong \sE({w}) \oplus \sG^{\prime}.
\end{equation*}
Hence
\begin{equation*}
{\pi_s}_*\pi_s^* \sE({\overline{w}})[1] \cong \sE({\overline{w}})[1] \oplus \sE({\overline{w}})[-1] \oplus \sG \oplus {\pi_s}_* \sG^{\prime}.
\end{equation*}
However, because $h\h_s = (v+v^{-1})h$ for any $h \in \He_s$, Lemma
\ref{lem-parchar} yields
\begin{equation*}
\ch({\pi_s}_*\pi_s^* \sE(\overline{w})[1] ) = (v + v^{-1}) \ch( \sE(\overline{w}))
\end{equation*}
and so $\ch(\sG) = \ch(\pi_{s*} \sG') = 0$. Hence
$\sG$ and $\sG'$ are zero.
\end{proof}


In $D^b_{\Lambda}(G/B)$ and $D^b_{\Lambda}(G/P_s)$ there exist the
middle perversity intersection cohomology
sheaves $\IC(w, \kk)$ and $\IC(\overline{w}, \kk)$ (see \cite{BBD, decperv}). The intersection cohomology complex
$\IC(w,\kk)$ is determined up to (canonical) isomorphism by the following conditions:
\begin{enumerate}
\item[IC1)] $i_x^* \IC(w,\kk) = 0$ for $x \not < w$;
\item[IC2)] $i_w^* \IC(w,\kk) \cong \underline{\kk}_w[\dim X_w]$;
\item[IC3)] $\mathcal{H}^j(i_x^* \IC(w,\kk)) = 0$ for $x < w$ and $j
  \ge -\dim X_x$;
\item[IC4)] $\mathcal{H}^j(i_x^! \IC(w,\kk)) = 0$ for $x < w$ and $j
  \le -\dim X_x$.
\end{enumerate}
Entirely analogous conditions define $\IC(\overline{w}, \kk)$. 
If $\kk$ is a field then the basic properties of the intersection
cohomology complexes are discussed in \cite{BBD}. If $\kk$ is not a
field, then there are (at least) two choices of what one means by the
intersection cohomology complex with coefficients in $\kk$; the definition given above
corresponds to the choice of perversity $p$ rather than $p_+$ (in the
notation of \cite{decperv}). We will never need the intersection cohomology
complex corresponding to $p_+$.
If $\kk$ is a field then the
intersection cohomology complexes  $\IC(w,\kk)$ are (Verdier)
self-dual. This is no longer true in general if $\kk$ is a principal ideal
domain.

The intersection cohomology complexes admit equivariant lifts
$\IC_B(w, \kk) \in D^b_B(G/B)$ and $\IC_B(\overline{w}, \kk) \in D^b_B(G/P_s)$  which are uniquely determined up to
isomorphism by requiring that their image under the forgetful functor
is the corresponding non-equivariant intersection cohomology complex (see
\cite[5.2]{BLu}). Equivalently, $\IC_B(w, \kk)$ is the uniqe object
satisfying the equivariant analogues of IC1), IC2), IC3) and IC4)
above (where we replace $\IC(w,\kk)$ by $\IC_B(w,\kk))$
throughout). As with parity sheaves, we write $\IC(w)$, $\IC_B(w)$,
$\IC(\overline{w})$ and $\IC_B(\overline{w})$ instead of $\IC(w,\kk)$
etc. if the ring of
coefficients $\kk$ is clear from the context.

The first relationship between parity sheaves and intersection
cohomology complexes is the following:






\begin{prop} If $\kk$ is a field of characteristic 0 then
  $\sE(w) \cong \IC_B(w)$. \label{prop:parity=IC}
\end{prop}

\begin{proof}
  The intersection cohomology complexes $\IC_B(w)$ are simple objects in
  the heart of the perverse $t$-structure on $D^b_B(G/B)$ (see \cite{BBD}) and are
  therefore indecomposable. Hence we will be done by Theorem
  \ref{thm-clas} if we can show that $\IC_B(w)$ is a parity sheaf.

Because $\kk$ is a field of characteristic
zero, $\IC_B(w)\vartheta_s$ is isomorphic
to a direct sum of shifts of intersection cohomology complexes for all
$w \in W$ (by the
  decomposition theorem (see \cite[6.2.5]{BBD} or
\cite[2.1.1]{dCM2}) and \cite[5.3]{BLu}) together with the fact that
$\pi_s$ is a smooth fibration). Moreover it is easy to see that if $ws
< w$ then $\IC_B(w)$ is a direct summand of $\IC_B(ws)
\vartheta_s$. Clearly $\IC_B(id) $ is parity, and
hence all $\IC_B(w)$ are parity by induction.
\end{proof}

Now let $\F \subset \kk$ denote a
finite subfield of $\kk$ and let $p$ denote the characteristic of both
fields. Let $\O$ denote a complete discrete valuation
ring of characteristic 0 with residue field $\F$ and field of fractions $\K$.
One has functors of extension of scalars:
\begin{gather*}
(- )\otimes_{\O}^L \F : D^b_B(G/B, \O) \to D^b_B(G/B, \F)  \\
(- )\otimes_{\F} \kk : D^b_B(G/B, \F) \to D^b_B(G/B, \kk)   \\
(- )\otimes_{\O} \K : D^b_B(G/B, \O) \to D^b_B(G/B, \K)   
\end{gather*}
By Lemma 2.19 and Proposition 2.22 in \cite{JMW2} one has:
\begin{gather}
\label{eq:parity Qp}
\sE(y, \O) \otimes_{\O} \K \text{ and } \sE(y, \O)
\otimes_{\O}^L \F \text{ are parity sheaves;} \\
  \label{eq:modreduc} 
 \sE(y, \O) \otimes_{\O}^L \F \cong \sE(y, \F).
\end{gather}

\begin{lem} \label{lem:ch=} For all $y \in Y$ we have
\[
\ch(\sE(y, \F)) = \ch(\sE(y, \O)) = \ch( \sE(y, \O) \otimes_\O \K).
\]
Similarly, if $\sG$ is a $*$-parity sheaf with coefficients in $\F$
then 
\[
\ch(\sG) = \ch(\sG \otimes_\F \kk).\]
\end{lem}

\begin{proof}
Fix $x \in W$. By definition of $*$-parity we can write $i^*_x \sE(y, \O) = V(x)
\otimes \underline{\O}_x$ for some finitely generated
  graded free $\O$-module $V(x)$. Because the
  stalks of $\underline{\O}_x$ 
are free\footnote{By the stalk of an equivariant sheaf or complex on $X$ we mean the
  stalk of the corresponding sheaf or complex on the Borel construction of $X$, see \cite{BLu}.}
 (and hence flat) and $i^*$
  commutes with extension of scalars we have isomorphisms
\begin{gather*}
i^*_x(\sE(y, \O) \otimes^L_{\O} \F) \cong (V(x) \otimes_{\O} \F) \otimes_{\F}
\underline{\F}_x, \\
i^*_x(\sE(y, \O) \otimes^L_{\O} \K) \cong (V(x) \otimes_{\O} \K) \otimes_{\K}
\underline{\K}_x.
\end{gather*}
Hence $\ch (\sE(y, \O) \otimes_{\O} \F) = \ch (\sE(y, \O)) =
\ch (\sE(y, \O) \otimes_{\O} \K)$. The first statement now follows
from \eqref{eq:modreduc}. The proof of the second statement is
entirely analogous.
\end{proof}

\begin{lem} \label{lem:parityext}
For all $y \in W$ we have
\[
\sE(y, \F) \otimes_\F \kk \cong \sE(y,\kk).
\]
\end{lem}

\begin{proof}
Firstly,  $\sE(y,\F)\otimes^L_\F \kk$ is parity by \cite[Lemma
2.19]{JMW2} and so we have to show that $\sE(y,\F)\otimes_\F \kk$ is 
indecomposable. Using \cite[Proposition 2.4]{JMW2} one deduces
easily that if we set $A = \End(\sE(y, \F))$ then
$\End(\sE(y,\F) \otimes^L_\F \kk) \cong A
\otimes_\F \kk$. Hence we will be done if we can show that $A
\otimes_\F \kk$ is a local ring which is the case if 
$A/\rad A \cong \F$ (where $\rad A$ denotes the Jacobsen radical of
$A$). However by  \cite[Corollary 2.6]{JMW2} we have a surjection
\[
\phi: A = \End(\sE(y, \F)) \twoheadrightarrow \End(i_y^* \sE(y, \F)) = \F
\]
and therefore $\rad A = \ker \phi$ because $A$ is a local ring. It follows that
$A/\rad A \cong \F$ as claimed.
\end{proof}

It follows that the characters of the indecomposable parity sheaves
only depend on the characteristic:

\begin{cor} \label{cor:parityextch}
For all $y \in W$ we have
\[ \ch(\sE(y, \F)) = \ch(\sE(y, \kk)). \]
\end{cor}

A simple and useful consequence of the above results is the following:

\begin{thm} \label{thm:Ncomb}
For any $y \in W$ we have
\[
\ch(\sE(y, \kk)) = \sum_{x\in W} Q_{xy} \h_x
\]
for some Laurent polynomials $Q_{xy} \in \N[v,v^{-1}]$.
\end{thm}

\begin{proof} By Corollary \ref{cor:parityextch} it is enough to prove
  the result for $\ch(\sE(y, \F))$. By Lemma \ref{lem:ch=} we have
\[
\ch (\sE(y, \F)) = \ch (\sE(y, \O)) = \ch (\sE(y, \O) \otimes_{\O}
  \K).
\]
Now \eqref{eq:parity Qp}  together with Theorem \ref{thm-clas} allows
us to write
\[
\sE(y, \O) \otimes_{\O}
  \K \cong \bigoplus_{x \in W} V(x,y) \otimes \sE(x, \K)
\]
for some finite dimensional graded vector spaces $V(x,y) = \bigoplus
V(x,y)_i$. Hence we have
\[
\ch(\sE(y, \F)) = \ch(\sE(y, \O) \otimes_{\O}^L \F) = \sum_{x
  \in W} Q_{xy} \ch(\sE(x,\K))
\]
where $Q_{xy} = \sum v^i\dim V(x,y)_i$. The result then follows from
Propositions \ref{prop:IC} and \ref{prop:parity=IC} which imply that
$\ch(\sE(x,\K)) = \h_x$.
\end{proof}

In the last part of this section we establish a relation between the
characters of the indecomposable parity sheaves with coefficients in
$\F$, and the presence of torsion in the cohomology groups of the
stalks or costalks of intersection cohomology complexes over $\O$ or
$\Z$.

In the following, when we say that the stalk of a sheaf (= complex)
is torsion free or has torsion, we mean that the corresponding
statement is true for the direct sum of the cohomology groups of the
stalk.\footnote{For example, the statement that a sheaf $\sF$ has torsion free stalks
means that all cohomology groups of the stalks of $\sF$ are torsion
free, whereas the statement the stalk of $\sF$ at $x$ has torsion
means that there exists a cohomology group of the stalk of $\sF$ at
$x$ which has torsion.}

\begin{prop} \label{prop:IC} 
For any $w \in W$, the following are equivalent:
  \begin{enumerate}
  \item $\IC(w, \O)$ has torsion free stalks and costalks;
  \item $\IC_B(w, \O) \cong \sE(w, \O)$;
  \item $\IC_B(w, \F) \cong \sE(w, \F)$;
  \item $\ch(\sE(w,\O)) = \ch(\sE(w, \F)) = \h_w$.
  \end{enumerate}
\end{prop}

\begin{proof} \emph{$(1) \Rightarrow (2)$}: Because
  $\K$ is flat as an $\O$-module $\IC(w, \O) \otimes_{\O} \K$
  satisfies the conditions IC1) -- IC4) and hence we have an
  isomorphism
\begin{equation} \label{eq:modext}
\IC(w, \O) \otimes_{\O} \K \cong \IC(w, \K).
\end{equation}
Now assume that $\IC(w, \O)$ has
 torsion free stalks and costalks and let $? \in \{ !, * \}$. Then
 $\mathcal{H}^j(i_x^? \IC(w, \O))$ is a local system with torsion free
 stalks, and hence is isomorphic to a direct sum of trivial local
 systems, because $X_x$ is simply connected. Now $\mathcal{H}^j$
 commutes with the forgetful functor, and hence $\mathcal{H}^j(i_x^?
 \IC_B(w, \O))$ is isomorphic to a direct sum of the
 $B$-equivariant constant sheaves $\underline{\O}_{x}$ (the
 forgetful functor from $B$-equivariant local systems on $X_x$ to
 local systems is fully faithful because $B$ is connected).
For all $x \in W$ we have isomorphisms
\[
\mathcal{H}^j(i_x^? \IC_B(w, \O)) \otimes_{\O} \K \cong 
\mathcal{H}^j(i_x^? (\IC_B(w, \O)\otimes_{\O} \K ))\cong 
\mathcal{H}^j(i_x^? \IC_B(w, \K))
\]
(we use that $i_x^?(-)$ commutes with extension of scalars) 
and the last group vanishes if $j \ne \ell(w) \textrm{ mod } 2$ by Proposition
\ref{prop:parity=IC}. Hence
\[
\mathcal{H}^j(i_x^? \IC_B(w, \O)) = 0 \quad \text{for $j \ne \ell(w)
  \textrm{ mod } 2$.}
\]
By a standard decalage argument (see \cite[Section 2.2]{JMW2}) it
follows that $i_x^? \IC_B(w, \O)$
is isomorphic to a direct sum of even shifts of the $B$-equivariant
constant sheaf $\underline{\O}_{x}[\ell(w)]$. In particular, $\IC_B(w,
\O)$ is a parity sheaf. Now, intersection cohomology complexes are
indecomposable (for example, because $\underline{\O}_w$ is indecomposable and
$(i_w)_{!*}$ is fully-faithful \cite[Proposition
2.29]{decperv}). Hence $\IC_B(w,\O) \cong \sE(w, \O)$ by Theorem
\ref{thm-clas}.

\emph{$(2) \Rightarrow (1)$}: Suppose $\IC_B(w, \O) \cong \sE(w,
\O)$ and choose $? \in \{ !, *\}$. By definition of $\sE(w,\O)$, $i_x^?
\sE(w,\O)$ is isomorphic to a direct sum of shifted $B$-equivariant
local systems $\underline{\O}_x$. Hence $i_x^? \sFor( \IC_B(w, \O))
\cong 
i_x^? \IC(w, \O)$ is also isomorphic to a direct sum of shifted local
systems with torsion free stalks. Hence the stalks and costalks of
$\IC(w, \O)$ are torsion free.

\emph{$(1) + (2)\Rightarrow (3)$}:
If the stalks of the cohomology sheaves of $i_x^?\IC(w,\O)$ are
torsion free, then we have
\[
\mathcal{H}^j(i_x^?\IC(w, \O) \otimes_{\O}^L \F) \cong 
\mathcal{H}^j(i_x^?\IC(w, \O)) \otimes_{\O} \F.
\]
Hence, $\IC(w, \O) \otimes^L_\O \F \cong \IC(w, \F)$. Hence
\[
\IC_B(w, \F) \cong \IC_B(w, \O) \otimes^L_\O \F \cong
\sE(w, \O) \otimes^L_\O \F \cong \sE(w, \F)
\]
with the last isomorphism following by \eqref{eq:modreduc}.

\emph{$(3) \Rightarrow (4)$}: Define
\[
h := \ch(\sE(w, \F)) \in \He.
\]
Then $h$ is a self-dual element of the Hecke algebra (by Proposition
\ref{cor-dual}). Let us write
\[
h = \sum_{x \in W} a_x H_x
\]
for some $a_x \in \N[v,v^{-1}]$. If $\sE(w,\F) \cong \IC_B(w,\F)$ then
we can translate the defining properties of the
intersection cohomology complex given above: IC1) gives
$a_x = 0$ for $x \not < w$; IC2) gives $a_w = 1$; and IC3) gives $a_x
\in v^{-1} \N[v^{-1}]$  for $x < w$. Hence $h = \h_w$ by the
uniqueness of the Kazhdan-Lusztig basis.

\emph{$(4) \Rightarrow (2)$}: By definition, $\sE(w, \O)$ is
self-dual, supported on the closure $\overline{X_w}$ and, for $? \in
\{ !, * \}$, $i_x^? \sE(w, \O)$ is isomorphic to
$\underline{\O}_w[\dim X_w]$ if
$x = w$ and is isomorphic to a direct sum of shifts of
equivariant constant sheaves in general. Using that
$\underline{\O}_x[\ell(x)]$ is self-dual and $\mathbb{D} i_x^* \cong i_x^!
\D$ we see that $\sE(w, \O)$ satisfies the above conditions to be
the (equivariant) intersection cohomology complex if it satisfies IC3) which is the
statement that, for all $x < w$,
\[
\mathcal{H}^j(i_x^* \sE(w, \O)) = 0 \quad \text{for $j \ge -\ell(x)$}.
\]
This is the case if and only if, when we write $\ch(\sE(w,
\O)) = \sum_{x \in W} b_x H_x$ we have $b_x \in v^{-1}\N[v^{-1}]$ for
$x < w$. However, this is clearly satisfied if $\ch(\sE(w, \O)) = \h_w$.
\end{proof}

The above proposition has an obvious analogue on $G/P_s$ for $s \in
S$. (One can check that the above proof also applies in this situation.)

\begin{prop} \label{prop:ICs}
Suppose that $w \in W$ with $ws < w$. Then the following
  are equivalent:
  \begin{enumerate}
  \item $\IC(\overline{w}, \O)$ has torsion free stalks and costalks;
  \item $\IC_B(\overline{w}, \O) \cong \sE(\overline{w}, \O)$;
  \item $\IC_B(\overline{w}, \F) \cong \sE(\overline{w}, \F)$;
  \item $\ch(\sE(\overline{w},\O)) = \ch(\sE(\overline{w}, \F)) = \h_w$.
  \end{enumerate}
\end{prop}

We finish this section with a corollary of the above results which
justifies a number of statements made in the introduction.

\begin{cor} \label{cor:Ztorsion} The following are equivalent:
  \begin{enumerate}
  \item the stalks and
  costalks of $\IC(x, \Z)$ are free of $p$-torsion;
\item $\sE(w,\kk) \cong \IC_B(w,\kk)$;
\item $\ch(\sE(w,\kk)) \cong \h_w$.
  \end{enumerate}
\end{cor}

\begin{proof} Consider the statements:
\begin{enumerate}
\item[(2')] $\sE(w,\F) \cong \IC_B(w,\F)$;
\item[(3')] $\ch(\sE(w,\F)) \cong \h_w$.
\end{enumerate}
Clearly (3) $\Leftrightarrow$ (3') by Corollary
\ref{cor:parityextch}. We first show (1) $\Leftrightarrow$ (2') $\Leftrightarrow$ (3') and
then (2) $\Leftrightarrow$ (2').

Because $\O$ is a flat $\Z$-module one has an isomorphism
\[
\IC(w, \O) \cong \IC(w, \Z) \otimes_{\Z} \O.
\]
Hence the stalks or costalks of $\IC(w, \O)$ will have torsion if and
only if the stalks or costalks of $\IC(w, \Z)$ have $p$-torsion. Hence
we have that (1), (2') and (3') are equivalent by Proposition
\ref{prop:IC}.

We can check whether $\sE(w,\F)$ is isomorphic to $\IC_B(w, \F)$
(resp. $\sE(w, \kk)$ is isomorphic to $\IC_B(w, \kk)$) by verifying
conditions IC1), IC20, IC3) and IC4) given at the start of this
section. As $\F \subset \kk$ is flat, it is easy to see that
these conditions will be verified for $\sE(w, \F)$ if and only if they
are verified for $\sE(w,\kk)$. Hence (2) and (2') are equivalent.\end{proof}



\section{Separated Elements} \label{sec-sep}

In this section $\kk$ denotes a field of charactheristic
$p$ and we fix a set of
representatives $\{ \sE(w) \; | \; w \in W \}$ for the
isomorphism classes of indecomposable parity sheaves on $G/B$ with
coefficients in $\kk$. We would like to investigate when their characters are equal to the
Kazhdan-Lusztig basis. By results of the previous section, the 
characters of indecomposable parity sheaves
yield a self-dual basis of $\mathcal{H}$ with certain positivity
properties which are shared by the Kazhdan-Lusztig basis. For example,
by the results of the previous section, if we express $\ch(\sE(x))\h_s$
and $\h_s\ch(\sE(x))$ in the basis $\{ \ch(\sE(w)) \;| \;w \in W \}$
then the coefficients belong to $\mathbb{N}[v,v^{-1}]$. In this 
section we investigate to what extent these properties already
determine the basis.\footnote{One can also prove that, if one expresses any
  product $\ch(\sE(x)) \ch(\sE(y))$ in the basis $\{ \ch(\sE(w)) \;|
  \;w \in W \}$ then the coefficients are Laurent polynomials with
  positive coefficients. We will not make use of this stronger positivity
  property below.
}

We start by introducing some notation. Given $a, a' \in \Z[v,v^{-1}]$
write
\[
a \le a' \Leftrightarrow a' - a \in \N[v,v^{-1}].
\]
Given $h, h' \in \He$ we expand them in the Kazhdan-Lusztig basis as $h
= \sum a_x \h_x$ and $h' = \sum a_x' \h_x$ and write
\[
h \le h' \Leftrightarrow a_x \le a'_x \text{ for all $x \in W$.}
\]
Using the fact that $0 \le \h_w\h_s$ and $0 \le \h_s\h_w$ for all $w \in W$ and $s
\in $ (which follows from \eqref{KL-mult-left} and  \eqref{KL-mult-right} in Section \ref{sec-hecke}) we have 
\begin{equation} \label{eq:hs}
h \le h' \Rightarrow h\h_s \le h'\h_s \text{ and } \h_sh \le \h_sh'.
\end{equation}

Given polynomials $ a^i = \sum_j
a^i_jv^j  \in \Z[v,v^{-1}]$ where $i$ runs over some finite indexing set $I$ define
\[
\min_{i \in I} \{ a^i \} = \sum_j (\min_{i \in I} \{ a^i_j \} )v^j.
\]
For example $\min \{ v^{-2} + 1 + 2v^2, 4v^{-2} + v^2 \} =
v^{-2} + v^2$. Similarly, if $h^i = \sum_{x \in W} a^i_x \h_x$ is a family of
elements of $\He$ indexed by $i \in I$ (again assumed finite) set
\[
\min_{i \in I} \{ h^i \} = \sum_{x \in W} \min_{i \in I} \{ a^i_x \} \h_x.
\]
Clearly if $a \in \Z[v,v^{-1}]$ satisfies $a \le a^i$ for all $i \in
I$ then $a \le \min_{i \in I} \{ a^i
\}$. Similarly, if $h \in \He$ satisfies $h \le h^i$ for all $i
\in I$ then $h
\le \min_{i \in I} \{ h^i \}$.

We now define elements $\bo_y \in \He$ for each $y \in W$. We define
these elements inductively as
follows:
\begin{enumerate}
\item $\bo_{id} = \h_{id} = H_{id}$;
\item Let $y \in W$ and assume that we have defined $\bo_x$ for all $x
  <  y$. Define
\[
\bo_y = \min \{ \min_{s \in \dL(y) } \{ \h_s \bo_{sy} \} ,  \min_{t \in
  \dR(y) } \{ \bo_{yt} \h_t  \} \}.
\]
\end{enumerate}
For example $\bo_s = \h_s$ for all $s \in S$ and $\bo_{st} = \h_{st}$
for any $s \ne t \in S$.

By induction, all the information needed to evaluate this formula is
encoded in the $W$-graph (using \eqref{KL-mult-left} and
\eqref{KL-mult-right}  in Section 
\ref{sec-hecke}). Another immediate consequence of this formula is that
if we define polynomials $D_{xy} \in \Z[v,v^{-1}]$ by
\[
\bo_y = \sum_x D_{xy} \h_x
\]
then $D_{xy} \in \N[v,v^{-1}]$,  $D_{xy} = 0$ unless $x \le y$ and $D_{yy} = 1$. In particular, 
$\{ \bo_y \; | \; y \in W \}$ is a basis of $\He$.

\begin{defi}
  Let $y \in W$, if $\bo_y = \h_y$ we say that $y$
  is \emph{separated}. We denote the set of separated elements in $W$
  by $\sigma(W)$.
\end{defi}

We will see in Proposition \ref{prop-sup} below that for any separated
$y \in W$ we have that $\ch(\sE(y)) = \h_y$ for all fields of
coefficients $\kk$. This is the content of Theorem \ref{thm-sep}
of the introduction.

\begin{ex} \label{ex-dih}
Let $W$ be a dihedral group:
\begin{displaymath}
D_n = \langle s, t \; | \; s^2 = t^2 = (st)^n = id \rangle.
\end{displaymath}
Then one may show that $\underline{B}_x = \h_x$ if and only if $x \in \{ id, s, t, st, ts, w_0 \}$. In
particular, $A_2$ and $A_1 \times A_1$ are the only rank two Weyl
groups in which $\sigma(W) = W$.
\end{ex}

With the above notions in hand, we now revisit the problem of
calculating the characters of indecomposable parity sheaves.

\begin{lem} \label{lem:le}
Let $\sF$ be a parity sheaf. Then for any direct summand
  $\sG$ of $\sF$ one has
  \[
\ch(\sG) \le \ch(\sF).
\]
\end{lem}

\begin{proof} This is immediate from the fact that, by Theorems
  \ref{thm-clas} and \ref{thm:Ncomb}, the
  characters of parity sheaves are $\N[v,v^{-1}]$-linear
  combinations of the Kazhdan-Lusztig basis $\{ \h_x \; | \; x \in W
  \}$.
\end{proof}

\begin{prop} \label{prop-sup} With coefficients $\kk$ in
any field we have
\begin{displaymath}
\ch(\sE(y)) \le \bo_y
\end{displaymath}
for all $y \in W$.

In particular, for all $y \in \sigma(W)$ we have $\ch(\sE(y)) =
\h_y$.
\end{prop}

\begin{proof} Clearly $\ch(\sE(id)) = \h_{id} = \bo_{id}$ and 
so we may assume by induction that the proposition is true for all $x
< y$. For all $s \in \dL(y)$ we know that $\sE(y)$ occurs as a
direct summand of $\ver_s \sE(sy)$. Now
\[
\ch(\sE(y)) \le \ch( \ver_s \sE(sy)) = \h_s \ch(\sE(sy)) \le \h_s\bo_{sy}.
\]
by Lemma \ref{lem:le}, Lemma \ref{lem-parchar} and \eqref{eq:hs} respectively.
Similarly, for all $t \in \dR(y)$,
$\sE(y)$ occurs as a direct summand of $\sE(yt) \ver_t$ and so
\[
\ch( \sE(y)) \le \bo_{yt} \h_t.
\]
Intersecting all of these conditions (see the remarks at the
beginning of this section) gives
\[
\ch(\sE(y)) \le \min \{ \min_{s \in \dL(y)} \{ \h_s \bo_{sy} \}, 
\min_{t \in \dR(y)} \{ \bo_{yt} \h_t \} \} = \bo_y.
\]
The last statement of the proposition is immediate: if $\bo_y = \h_y$ then
$0 \le \ch(\sE(y)) \le \bo_y$
forces $\ch(\sE(y)) = \h_y$ because $\ch(\sE(y))$ is non-zero. \end{proof}

\begin{remark} One may show that if $G$ is of rank 2 with Weyl
  group $W$ (a dihedral group), then the Schubert varieties corresponding
  to elements $w \in \sigma(W)$
  are smooth. Hence Proposition
  \ref{prop-sup} does not give any new information in this case.
\end{remark}

It remains to prove Theorem \ref{thm-dec} from the introduction. In
fact we prove a more precise result. Fix $w \in W$ and $s \in S$ with
$ws > w$. By \eqref{KL-mult-right} in Section \ref{sec-hecke} we can
write:
\begin{equation} \label{eq:multexpand}
  \h_w\h_s  = \h_{ws}  + \sum_{zs < z < w} \mu(z, w)\h_z.
\end{equation}
Define $A(w,s) = \{ ws \} \cup \{ z \in W \; | \; zs < z < w, \mu(z,w) \ne 0 \}$.

\begin{prop}
  Suppose that $w$ and each element of $A(w,s)$ belongs to
  $\sigma(W)$. Then one has an isomorphism
\[
\pi_{s*} \sE(w) \cong \sE(\overline{w}) \oplus \bigoplus_{zs < z < w}
\sE(\overline{z})^{\oplus \mu(z,w)}.
\]
\end{prop}

\begin{proof}
  By Theorem \ref{thm-clas} we have a decomposition
\[
\pi_{s*} \sE(w) \cong \bigoplus_{z \in W ; zs < z \atop \nu \in \Z}
\sE(\overline{z})[\nu]^{\oplus m_{z, \nu}}.
\]
for certain natural numbers $m_{z, \nu} \in \N$. Applying $\pi_s^*$
and using Proposition \ref{cor:Ps} we obtain
\[
\sE(w) \ver_s = \pi_s^* \pi_{s*} \sE(w)[1] \cong \bigoplus_{z \in W ;
  zs < z \atop \nu \in \Z}
\sE(z)[\nu]^{\oplus m_{z, \nu}}.
\]
By assumption, $w \in \sigma(W)$ and so $\ch(\sE(w)) = \h_w$. By Lemmas
\ref{lem:le} and \ref{lem-parchar} we have, for any $z$,
\[
\ch( \sE(z)[\nu]^{\oplus m_{z, \nu}}) \le \ch(\sE(w)\ver_s) = \h_w\h_s.
\]
In view of \eqref{eq:multexpand} we conclude that $m_{z, \nu} = 0$
unless $\nu =0$ and that $m_{z, \nu} \ne 0$ implies 
$z \in A(w,s)$. On the other hand, our assumptions guarantee that
$\ch(\sE(z)) = \h_z$ for all $z \in
A(w,s)$. Hence $m_{ws,0} = 1$ and $m_{z,0} = \mu(z,w)$ for $ws \ne z \in
A(w,s)$. The proposition then follows.
\end{proof}

If $w \in \sigma(W)$ then $\ch(\sE(w)) = \h_w$ and so
$\sE(w, \kk) \cong \IC_B(w, \kk)$ by Corollary \ref{cor:Ztorsion}.
Hence (under the
same assumptions as in the above proposition) we have an isomorphism
\[
\pi_{s*} \IC_B(w) \cong \IC_B(\overline{w}) \oplus \bigoplus_{zs < z < w}
\IC_B(\overline{z})^{\oplus \mu(z,w)}.
\]
Applying the forgetful functor yields
\[
\pi_{s*} \IC(w) \cong \IC(\overline{w}) \oplus \bigoplus_{zs < z < w}
\IC(\overline{z})^{\oplus \mu(z,w)}.
\]
On the other hand, if $ws < w$ then, by Proposition \ref{cor:Ps}, 
we have
\[
\pi_{s*} \sE(w) \cong \sE(\overline{w})[1] \oplus \sE(\overline{w})[-1].
\]
Now, if $w \in \sigma(W)$, then $\sE(w) \cong \IC_B(w)$ as above and
\[
\h_w = \ch(\sE(w)) = \ch(\pi_s^*\sE(\overline{w})[1]) = \ch(\sE(\overline{w}))
\]
by Propositions \ref{prop-sup} and \ref{cor:Ps} and Lemma
\ref{lem-parchar}. Proposition \ref{prop:ICs} then gives
$\sE(\overline{w}) \cong \IC_B(\overline{w})$. Hence
\[
\pi_* \IC(w) \cong \IC(\overline{w})[1] \oplus \IC(\overline{w})[-1].
\]
This proves Theorem \ref{thm-dec} from the introduction.

\section{Results of Computer Calculations} \label{sec-comp}

In this section we give some examples of the sets $\sigma(W) \subset
W$ and the basis $\{ \underline{B}_x \}$ for low rank Weyl groups. As
is clear from its definition,
the only information needed to calculate the basis $\{ \underline{B}_x \}$
is the Coxeter system $(W,S)$ and its $W$-graph described in Section \ref{sec-hecke}.
 However, no general description of
the $W$-graph is known  (for descriptions of some subgraphs see
\cite{LaSch} and \cite{Ke} and for a description of the computational
aspects of the problem see \cite{dCl2} and \cite{OK}).

Thus, in order to calculate the basis $\{ \underline{B}_x \}$ (and
hence $\sigma(W)$) we have to restrict ourselves to examples. This involves two steps:
\begin{enumerate}
\item calculation of the $W$-graph of $(W,S)$, and
\item calculation of the basis $\{ \underline{B}_x \}$ using the $W$-graph.
\end{enumerate}
Step 1) is computationally quite difficult, especially when the Weyl group is large. Luckily there exists the program \emph{Coxeter} written by Fokko du Cloux \cite{dCl1}, which calculates the $W$-graph very efficiently. Step 2) is then relatively straightforward. A crude implementation in Magma (whose routines for handling Coxeter groups proved very useful) as well as the $W$-graphs obtained from \emph{Coxeter} are available at:
\begin{center}
\url{http://people.maths.ox.ac.uk/williamsong/torsion/}
\end{center}
This site also contains a complete description of the basis
elements $\underline{B}_x$ for $x \in W$
and sets $\sigma(W)$ for all Weyl groups of ranks less than 6.

We will now describe examples of the sets $\sigma(W)$.

\subsection{$A_n, n \le 6$} Here $\sigma(W) = W$. Thus, all
intersection cohomology complexes with coefficients in an arbitrary
field have the same characters as in
characteristic zero and the decomposition theorem is true with field
coefficients of any characteristic. This is the statement of Theorem \ref{thm:a6}.

\subsection{$A_7$}\label{subsec:A7}
Let $W = A_7$ with Coxeter generators $s_i$ with $i \in \{ 1, \dots, 7
 \}$ corresponding to the simple transpositions $(i, i+1)$.  In $W$, 38 of the 40 320 elements do not belong to $\sigma(W)$. The elements which do not lie in $\sigma(W)$ break up naturally into five groups, which we now describe.
 
Consider the following elements of $W$:
\begin{equation*}
\begin{array}{c}
w_1 =  \begin{array}{c} 46718235 \\
 \includegraphics[width=2cm]{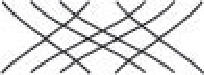} \end{array} ,
 \quad w_2 = 
 \begin{array}{c} 67823451 \\
 \includegraphics[width=2cm]{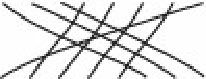} \end{array}, \quad
 w_3 = 
 \begin{array}{c} 84567123 \\
 \reflectbox{\includegraphics[width=2cm]{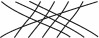}}
 \end{array}
 \\
w_4 = 
\begin{array}{c}
62845173 \\
 \includegraphics[width=2cm]{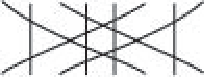} \end{array}
\; \text{ and } \; w_5 = 
\begin{array}{c}
84627351 \\
 \includegraphics[width=2cm]{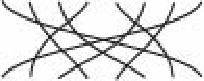} \end{array}.
 \end{array}
 \end{equation*}
 
The first group consists of
\begin{equation*}
K_1 = \{ uw_1v \; | \; u, v \in \langle s_4 \rangle \}.
\end{equation*}


The second group consists of
\begin{equation*}
K_2 = \{ w_2, s_5w_2, s_1w_2, s_1s_5w_2, w_2s_3, w_2s_7, w_2s_3s_7 \}.
\end{equation*}
(note that $w_2$ is maximal in $K_2$). The third group $K_3$ is obtained from $K_2$ by inversion (or by applying the automorphism $s_i \mapsto s_{8-i}$). It contains $w_3$ as a maximal element. The fourth group consists of
\begin{equation*}
K_4 = \{ uw_4v \; | \; u, v \in \langle s_2, s_6 \rangle \}.
\end{equation*} The fifth group consists of
\begin{equation*}
K_5 = \{ uw_5v \; | u, v \in \langle s_4 \rangle \}.
\end{equation*}
It would be interesting to investigate the intersection cohomology complexes corresponding to the minimal elements in $K_i$ for $i \ne 1$ directly.

Note that the set $K_1$ has already arisen in Kazhdan-Lusztig
combinatorics; these are the so-called ``hexagon permutations'' of
Billey and Warrington (see \cite{BW}, the name refers
to a characteristic hexagon shape appearing in their heap 
representation).

The significance of these permutations is explained by the following result.
We say that a permutation $w\in S_n$ contains the pattern of a permutation
$y\in S_m$ if there is a collection of indices $1 \le i_1 <\dots < i_m \le n$
so that $w(i_1),\dots,w(i_m)$ are in the same relative order as 
$y(1),\dots,y(m)$.  Otherwise, we say that $w$ avoids the pattern $y$. 

\begin{thm}(\cite[Theorem 1]{BW}) Let $\bw$ be a reduced word for an element $w\in S_n$.
Then the corresponding Bott-Samelson resolution $\pi\colon \Xt_\bw \to
X_w$ (see the appendix) is a small map if and only if
the permutation $w$ avoids the pattern $321$ and the four hexagon 
permutations $w_1$, $s_4w_1$, $w_1s_4$ and $s_4w_1s_4$.
\end{thm}

For a more general notion of pattern avoidance which works for general
Coxeter groups, see \cite{BiP}.  A geometric interpretation of pattern avoidance is given in \cite{BiBr}.  

The Bott-Samelson resolutions of the hexagon permutations are semi-small,
but not small (as can be checked directly). In the appendix, Braden
treats the intersection cohomology
complex over $\Z$ corresponding to $w \in K_1$ in detail, and shows
that they have 2-torsion in their costalks.
In particular, $\ch(\sE(w,\kk)) \ne \h_w$ for $w \in K_1$ if $\kk$
is a field of characteristic 2 by Corollary \ref{cor:Ztorsion}.

\subsection{$B_2$ and $B_3$} As $B_2$ is dihedral, it has already
been covered in Example \ref{ex-dih} where we saw that, if $W = \langle s, t \; | \;
(st)^4 = 1 \rangle$ then $W \setminus \sigma(W) = \{ sts, tst \}$. One
calculates easily that
\[
\underline{B}_{sts} = \h_{sts} + \h_{s} \quad \text{and} \quad
\underline{B}_{tst} = \h_{tst} + \h_{t}.
\]
One may show that, if $s$ (resp. $t$) denotes the simple reflection
corresponding to the long (resp. short) simple root then the Schubert
variety corresponding to $sts$ is smooth and so $\sE(sts) \cong
\IC_B(sts)$ for any field of coefficients. Furthermore one may show (either by direct
calculation or using the results of \cite{JW}) that
$\sE(tst) \cong \IC_B(tst)$ if and only if the coefficients $\kk$ are
not of characteristic 2. It follows from Proposition \ref{prop-sup}
that, with coefficients of characteristic 2, one has
\[
\ch (\sE(tst) ) =  \h_{tst} + \h_{t}.
\]

Now consider $W = B_3$ with generators
\begin{displaymath}
\xymatrix{ s \ar@{-}[r] & t\ar@2{-}[r]  & v}.
\end{displaymath}
In this example, $\underline{B}_x$ is an
$\N$-linear combination of Kazhdan-Lusztig basis elements for all $x
\in W$ except $x = stsuts$, where one has
\[
\underline{B}_{stsuts} = \h_{stsuts} + (v+v^{-1})\h_{sts}.
\]
Of the 48 elements of $W$, 21 do
not lie in $\sigma(W)$. They are:
\begin{equation*}
\begin{array}{c} 
utu,
tut,
utsu,
tuts,
utsut,
tsuts,
sutu,
tsutu, \\
utsutu,
(tsu)^2,
tutsutu,
stut,
stsut,
stsuts,
(sut)^2,
sutsutu, \\
tsutsut,
tsutsutu,
stuts,
stutsutu,
stsutsut.
\end{array}
\end{equation*}
 
 \subsection{$B_4$ and $B_5$} In $B_4$, which contains 384 elements,
 221 elements do not lie in $\sigma(W)$. In $B_5$, which contains 3840
 elements, 2627 elements do not lie in $\sigma(W)$.

\subsection{$C_2$, $C_3$, $C_4$ and $C_5$} Note that this is covered
by the discussion of type $B$ above.

\subsection{$D_4$}\label{subsec:D4}
We label our generators $s$, $t$, $u$ and $v$ of $W$ as follows:
\begin{displaymath}
\xymatrix@R=0.15cm@C=0.15cm{
  &   & u\ar@{-}[dl] \\
s \ar@{-}[r]& t\\
  &   & v\ar@{-}[ul]}
\end{displaymath}
Here, 7 elements do not belong to
$\sigma(W)$. Let $\tau$ be the automorphism of $W$ mapping $s \mapsto
u \mapsto v \mapsto s$. The elements not in $\sigma(W)$ are $w_1 =
tvtsutv$, $\tau(w_1)$ and $\tau^2(w_1)$ as well as $w_2 = suvtvsu$,
$tw_2$, $w_2t$ and $tw_2t$. In the appendix, Braden discusses the case
of $w_2$ in more detail.

\subsection{$D_5$ and $D_6$} In $D_5$, 171 of the 1920 elements in $W$
do not lie in $\sigma(W)$. In $D_6$, which contains 23040 elements,
3713 elements of $W$ do not lie in $\sigma(W)$.

\subsection{$F_4$} In $F_4$, 949 of the 1152 elements do not lie in $\sigma(W)$.

\subsection{$G_2$} In this case we have already calculated $\sigma(W)$
in Example \ref{ex-dih}. Here we obtain nothing new. If $W = \langle
s, t | s^2 = t^2 = (st)^6 = 1 \rangle$ then $\sigma(W) = \{ 1, s, t,
st, ts, ststst \}$. However these Schubert varieties are smooth (this
is obvious for $X_{ststst} = G/B$ and for the others it follows
because the Bott-Samelson resolution is an isomorphism), and
so $\ch(\IC(w,\kk))= \ch(\sE(w,\kk)) = \h_w$ in any characteristic if $w
\in \sigma(W)$.

\subsection{Further calculations}

Let us briefly describe how the above algorithm may be taken further
with some of the geometric input contained in the appendix.

For example, if $W$ is of type $A_7$ then, with notation as in
\ref{subsec:A7}, the calculations in
the appendix allow one to conclude that $\ch(\sE(w_1)) = \h_{w_1}$
if $\Char \kk \ne 2$. Then the above algorithm may be used to deduce
that $\ch(\sE(w)) = \h_w$ for all hexagon permutations $w$.

Similarly in $D_4$ if one knows, with notation as in \ref{subsec:D4},
that $\ch(\sE(w_2)) = \h_{w_2}$
then it follows that $\sE(tw_2)$, $\sE(w_2t)$ and $\sE(tw_2t)$ all
have characters are given by Kazhdan-Lusztig basis elements. By the
calculations of the appendix this occurs if and only if $\Char \kk \ne 2$.

One can sometimes turn these arguments around to deduce that $\ch(\sE(w))
\ne \h_w$ for all $w \in K$ for some subset $K \subset W$,  once one knows
that this is the case for one $w \in K$. For example,
if $W$ is of type $D_4$ and one knows that $\ch(\sE(w_2)) \ne
\h_{w_2}$ (as is the case in characteristic 2) one cannot have
$\ch(\sE(w)) = \h_{w}$ for $w \in \{ tw_2, w_2t, tw_2t \}$. Indeed,
let us assume, for example, that $\ch(\sE(tw_2)) = \h_{tw_2}$. Using the
$W$-graph one may calculate
\[
\h_s\h_{tw_2} = \h_{stsuv} + \h_{sutvtsu} + \h_{stsvtsu} +
\h_{stsutsv} + \h_{stsutvtsu}. 
\]
Now, $stw_2 = stsutvtsu$ is separated, and hence $\ch(\sE(stw_2)) =
\h_{stw_2}$. It follows that we have a decomposition
\[
\ver_s(\sE(tw_2)) = \sE(stw_2) \oplus \sE'
\]
with $\sE(w_2)$ occuring
as a direct summand of $\sE'$. It follows by Lemma \ref{lem:le} that
\[
\ch(\sE(w_2)) \le \h_{stsuv} + \h_{w_2} + \h_{stsvtsu} +
\h_{stsutsv}.
\]
However, by Proposition \ref{prop-sup},
\[
\ch(\sE(w_2) \le \underline{B}_{w_2} = \h_{w_2} + \h_{suv}.
\]
Hence $\ch(\sE(w_2)) = \h_{w_2}$ as well. Analogous arguments apply in
the other cases.

Similarly one may show in type $A_7$ that if
$\ch(\sE_w) \ne \h_w$ for some hexagon permutation $w$, then
$\ch(\sE_x) \ne \h_x$ for all hexagon permutations. A geometric
explanation for this is given in Remark \ref{rem:otherhex} of the
appendix.

\appendix
\section{Examples of torsion in $IH$ of Schubert varieties in types $A_7$ and $D_4$}

\centerline{ \emph{by Tom Braden}
\footnote{supported in part by National Science Foundation grant DMS-0201823
and NSA grant H98230-08-1-0097}
}
\hspace{1cm}

 Let $G$ be a connected reductive complex group, and fix a choice of a 
Borel subgroup $B$, maximal torus $T$, and opposite Borel subgroup $B^-$,
so $T = B \cap B^-$.  
Let $\Phi^+ \subset \Phi \subset X(T)$
be the corresponding sets of roots and positive roots, chosen so that
the weights of $\Ad T$ acting on $\mathfrak b$ are $-\Phi^+ \cup \{0\}$.
Let $W = N(T)/T$ denote the Weyl group and $S \subset W$ the set of simple 
reflections.  Let $\ell\colon W \to \N$ be the length function, and 
$\le$ the Bruhat-Chevalley order on $W$.

Consider the flag variety $X = G/B$; $G$ acts on $X$ by
left multiplication.  The set of $T$-fixed points is in bijection 
with $W$ by $w \mapsto \wti B/B$, 
where $\wti$ is any lift of $w$ to $G$.
Using this bijection, we abuse notation and refer to points in $X^T$ and
elements of $W$ by the same symbols.

\smallskip The flag variety $X$ has two decompositions by Bruhat cells
and dual Bruhat cells $X = \coprod_{w\in W} X_w = \coprod_{w\in W} S_w$,
where $X_w = B\cdot w$ and $S_w = B^- \cdot w$.   The dual cell
$S_y$ is a normal slice at $y\in X_y$ 
to the stratification $\{X_w\}_{w \in W}$.  This is a consequence of
the following lemma, which we will also use in our analysis of Bott-Samelson
varieties.

Let $N$, $N^-$ be the unipotent parts of $B$, $B^-$.  Take $y \in W$, 
and define $N_y = yN^-y^{-1}$ and $N^+_y = N_y \cap N$.  
Then $N^+_y$ is a connected unipotent group of dimension $\ell(w)$.

\begin{lem} \label{equivariance lemma}
The map 
\[N_y^+ \times S_y \to X, \;\;(g, x) \mapsto g\cdot x \]
is an $N_y^+$-equivariant isomorphism onto a 
$T$-invariant Zariski neighborhood of $y$.  This isomorphism trivializes 
the stratification by Schubert cells $X_w$ in the sense that
all strata of the induced stratification on $N_y^+ \times S_y$
are of the form $N_y^+ \times S$ 
for some $S \subset S_y$.  The set $N_y^+ \times \{y\}$ is 
one such stratum, and it is mapped isomorphically onto $X_y$.
\end{lem}

\begin{proof} 
The map $\phi\colon N_y \to X$ sending $g$ to $g \cdot y$ is an isomorphism
onto a $T$-invariant neighborhood $U$ of $y$.  It identifies $N^+_y$ with
$X_y$ and $N^-_y$ with $S_y$, where we put 
$N^-_y = N_y \cap N^-$.  Since the multiplication map $N_y^+ \times N_y^- \to N_y$
is an isomorphism of varieties, we get isomorphisms
\[N_y^+ \times S_y \cong N_y^+ \times N_y^- \cong N_y \cong U.\]  

The $N_y^+$-equivariance is obvious, and the remaining statements about the
stratification follow easily from this.
\end{proof}

It is easy to describe the one-dimensional $T$-orbits in $X$.
The closure of a one-dimensional orbit is a closed irreducible 
$T$-invariant curve;  we will refer to such curves as ``$T$-curves'' 
for short.
For any positive root $\mu \in \Phi^+$ (simple or not), let $s_\mu\in W$ 
denote the corresponding reflection.  

\begin{prop} \label{flag 1d orbits}
For any $\mu \in \Phi^+$ and any $w \in W$, there is
a unique $T$-curve $C$ which
contains $w$ and $ws_\mu$, and all $T$-curves are
of this form.  The weight of the action of $T$ on the tangent
space $T_{w}C$ is 
$w(\mu)$.
\end{prop}

Since $s_{w(\mu)} = ws_{\mu}w^{-1}$, the above formula for the tangent weight
can also be given, up to sign, by saying that if $C$ is the $T$-curve joining
$w$ and $s_{\mu'}w$ for some $\mu' \in \Phi^+$, then 
the $T$-weight of $T_wC$ is $\pm \mu'$.  The sign can then be specified by 
noting that the weight is in $\Phi^+$ if and only if $w \le s_{\mu'}w$.

\subsection*{The Bott-Samelson variety}
Let $\bal = (\alpha_1,\alpha_2,\dots,\alpha_l)$ be a 
sequence of simple roots, not necessarily distinct, and 
let $\bw = (s_1, s_2, \dots, s_l)$ be the corresponding sequence of simple
reflections: $s_i = s_{\alpha_i}$.  
Put $w = s_1s_2\cdots s_l$.  Then $\bw$ is 
a reduced word for $w$ if $\ell(w) = l$.

For each simple reflection $s_i$, let $P_i$ be the corresponding 
minimal parabolic containing $B$, whose Lie algebra is 
$\mathfrak{b} \oplus \mathfrak{g}_{\alpha_i}$.
The Bott-Samelson variety $\Xt_\bw$ associated to $\bw$ is defined 
to be the quotient
\[(P_1 \times P_2 \times \dots \times P_l)/B^l,\]
where $B^l = B \times \dots \times B$ acts on 
$P_1 \times P_2 \times \dots \times P_l$ on the right by
\[(x_1, x_2, \dots, x_l)\cdot (b_1, b_2, \dots, b_l) = 
(x_1b_1, b_1^{-1}x_2b_2, \dots, b_{l-1}^{-1}x_lb_l).\]
Let $[x_1,\dots, x_l]$ denote the point of $\Xt_\bw$ corresponding
to $(x_1,\dots, x_l) \in P_1 \times \dots \times P_l$.

Let $B$ act on $\Xt_\bw$ by \[b\cdot[x_1,\dots,x_l] = [bx_1,x_2,\dots,x_l].\]
For any $1 \le k \le l$, define a $B$-equivariant map $\pi_k\colon \Xt_\bw \to X$ by
$\pi_k([x_1,\dots, x_l]) = x_1\cdots x_kB$.  Let $\pi = \pi_l$.

We can filter $\Xt_\bw$ by smaller Bott-Samelson varieties as follows.
For any subset $I \subset \{1, \dots l\}$, the $B$-invariant subvariety
\[\{[x_1, \dots, x_l] \in \Xt \mid x_i = 1\; \text{if $i \notin I$}\}\]
is clearly isomorphic to $\Xt_{\bw_I}$, where the
word $\bw_I$ is given by the simple reflections $s_i,\; i\in I$ 
taken in order of increasing $i$.  

\begin{remark} There is a decomposition of $\Xt_\bw$ into 
$2^l$ cells so that all of the subvarieties $\Xt_{\bw_I}$
are closures of cells.  However, this decomposition is 
not compatible with the projection $\pi$: the inverse 
image $\pi^{-1}(S_w)$ is not generally a union of cells.
It will be more useful for us to
consider a different paving of $\Xt_\bw$ by Bia\l ynicki-Birula cells below.
\end{remark}

The following easy result will be useful when we analyze the fibers of
the Bott-Samelson map.

\begin{lem} \label{transverse} For any $y \in Y$ the inverse image $\pi^{-1}(S_y)$
of the dual cell through $y$ 
is smooth.  In addition, if $I \subset \{1,\dots,l\}$, then
$\Sigma_{\bw_I}$ intersects $\pi^{-1}(S_y)$ transversely.
\end{lem}
\begin{proof} Lemma \ref{equivariance lemma} implies that
the action map $N^+_y \times \pi^{-1}(S_y) \to \Sigma_\bw$ is
an isomorphism onto an open subset of $\Sigma_\bw$.  It
follows that $\pi^{-1}(S_y)$ is smooth and meets any
smooth $N^+_y$-invariant subvariety transversely.
\end{proof}

\subsection*{Fixed points} We next describe the set $(\Xt_\bw)^T$ of $T$-fixed points.  
Let $D_l = (\Z/2\Z)^l = \{0,1\}^l$.  For each
$\vep = (\vep(1),\dots,\vep(l)) \in D_l$, define 
$p(\vep) = p_\bw(\vep)\in \Xt_\bw$ and $\bw^\vep \in W$ by
\[p(\vep) = [\sti_1^{\vep(1)},\dots,\sti_l^{\vep(l)}],
\; \bw^\vep = s_1^{\vep(1)}\cdots s_l^{\vep(l)}.\]
For any $\vep \in D_l$ and any 
$1 \le k \le l$, define $\vep[k] = (\vep(1),\dots,\vep(k),0\dots,0) \in D_l$.
As usual, we will refer to elements of $W$ and points
of $X^T$ by the same symbols. 

\begin{prop} \label{BS fixed points}
The map $\vep \mapsto p(\vep)$ is a bijection
between $D_l$ and $(\Xt_\bw)^T$. For any $\vep \in D_l$, we have
$\pi_k(p(\vep)) = \bw^{\vep[k]}$, and, in particular, 
$\pi(p(\vep)) = \bw^\vep$.
\end{prop}

\begin{ex} \label{SL3 example}  Let $G = GL(3, \C)$.  There are two simple roots, 
call them $\rho_1$ and $\rho_2$.

Let $\bw = (s_{\rho^{}_1},s_{\rho^{}_2},s_{\rho^{}_1},s_{\rho^{}_2},s_{\rho^{}_1})$.  Let 
$w_0 = s_{\rho^{}_1}s_{\rho^{}_2}s_{\rho^{}_1}$ denote the longest
element in $W$.  Then there are five $T$-fixed points in $\pi^{-1}(w_0)$,
namely $p(\vep)$, where 
\[\vep\in\{11100, 01110, 00111, 10011, 11001\}.\]
Let us denote these five elements of $D_5$ 
by $\vep_1,\dots, \vep_5$.  
There exists a $T$-curve containing
$p(\vep_i)$ and 
$p(\vep_j)$ if and only if $i = j \pm 1 \mod 5$.  
For $i = 1, \dots, 5$, the tangent
weight at $p(\vep_i)$ of the $T$-curve 
joining $p(\vep_i)$ to
$p(\vep_{i+1})$ is the $i$th element of the list 
$-\rho_1$, $-\rho_2$, $\rho_1$, $\rho_2+\rho_1$, and $\rho_2$.

The composition of  $\pi_2\colon \Sig_\bw \to X$ with the projection 
$X \to G/P_1 \cong \C\PP^2$ restricts to a birational map
$\pi^{-1}(w_0) \to \C\PP^2$ which identifies $\pi^{-1}(w_0)$
with the blow-up of $\C\PP^2$ at two points. 
The exceptional fibers are the $T$-curves joining $p(\vep_1)$ to 
$p(\vep_5)$ and $p(\vep_3)$ to $p(\vep_4)$. 

\end{ex}

\subsection*{One-dimensional orbits} The one-dimensional $T$-orbits of $\Xt_\bw$ are more
difficult to classify than the fixed points.  Unlike the flag variety $G/B$, 
Bott-Samelson varieties generally have infinitely many 
$T$-curves.  We will describe a
collection of $T$-curves which span the tangent space at
each fixed point, but there are in general many other $T$-curves. 

Denote the standard basis of $D_l$ by
$\delta_i$, where $\delta_i(j) = \delta_{ij}$ is the Kronecker 
$\delta$-function.  For any $\vep\in D_l$ and $1\le i\le k$,
we have a $T$-curve joining $p(\vep)$ and $p(\vep +\delta_i)$,
namely
\[\{[\sti^{\vep(1)}_1, \dots, \sti^{\vep(i-1)}_{i-1},x, \sti^{\vep(i+1)}_{i+1}, \dots,\sti^{\vep(l)}_{l} ] \mid x \in P_i\}.\]  
This curve 
projects under $\pi$ to the $T$-curve in $G/B$ which 
joins $\bw^\vep$ and $\bw^{\vep+\delta_i}$, and so the
tangent weight of this curve at
$p(\vep)$ is $\pm\bw^{\vep[i-1]}(\alpha_i)$.  Note that 
$T$-curves which project down to fixed points,
such as the ones in Example \ref{SL3 example},
are not of this type.

\subsection*{Bia\l ynicki-Birula cells}  
Besides their definition as $B$-orbits, 
the Bruhat cells $\{X_w\}$ in the flag variety $X$ can 
also be described as Bia\l ynicki-Birula cells for 
the action of a strictly dominant cocharacter $\zeta\colon \C^*\to T$.  
For any $w \in W$, we have
\[X_w = B\cdot w = \{x \in X\mid \lim_{t\to \infty} \zeta(t)\cdot x = w\}.\]

The Bott-Samelson variety $\Xt_\bw$ will not in general have finitely many
$B$-orbits, but we can still consider its Bia\l ynicki-Birula cells (for the 
same cocharacter $\zeta$).  Given $\vep\in D_l$, we define
\[\Xt_{\bw,\vep} = \{x \in \Xt_\bw\mid \lim_{t\to \infty} 
\zeta(t)\cdot x = p(\vep)\}.\]

\begin{thm}[\cite{G1,Hae}] \label{BB cells}
The dimension of the Bia\l ynicki-Birula cell $\Xt_{\bw,\vep}$ is 
\[
\begin{split}
\dim_\C \Xt_{\bw,\vep} = 
\#\{1 \le k \le l \mid \bw^{\vep[k]}(\alpha_k) \in -\Phi^+\}\\
= \#\{1 \le k \le l \mid \ell(\bw^{\vep[k]}) > \ell(\bw^{\vep[k]}s_k)\}.
\end{split}\]
The fibers of the map $\Xt_{\bw,\vep} \to X_w$,  $w = {\bw^\vep}$ are affine spaces of dimension
\[
\begin{split}
\dim_\C \Xt_{\bw,\vep}  \cap \pi^{-1}(w) = 
\#\{1 \le k \le l \mid \bw^{\vep[k-1]}(\alpha_k) \in -\Phi^+\}\\
= \#\{1 \le k \le l \mid 
\ell(\bw^{\vep[k-1]}) > \ell(\bw^{\vep[k-1]}s_k)\}.
\end{split}\]
The cells $\pi^{-1}(w) \cap \Xt_{\bw,\vep}$ for all $\vep$ with $w = \bw^\vep$ 
give a paving by affines of $\pi^{-1}(w)$.
\end{thm}

As a corollary we obtain the following relation
between the cells $\Sigma_{\bw,\vep}$ and the sub-Bott-Samelson
varieties $\Sigma_{\bw_I}$.  Take $\vep \in D_l$, and a subset $I \subset \{1, \dots, l\}$ 
and assume that $\vep(i) = 1$ implies $i \in I$ for all $i$, so 
the fixed point $p(\vep)$ lies in the subvariety $\Xt_{\bw_I} \subset \Xt_\bw$.

\begin{cor} \label{Containment of cells} The cell $\Sigma_{\bw,\vep}$ is contained in $\Xt_{\bw_I} \subset \Xt_\bw$
if and only if 
\[\ell(\bw^{\vep[i-1]}) < \ell(\bw^{\vep[i-1]}s_i) \; \text{for all $i \notin I$}.\] 
\end{cor}
\begin{proof} Using Theorem \ref{BB cells}, compare the dimension of 
$\Sigma_{\bw,\vep}$ with the dimension of the Bia\l ynicki-Birula cell
in $\Xt_{\bw_I}$ containing $p(\vep)$.
\end{proof}

\subsection*{An obstruction to splitting the Bott-Samelson sheaf} 
Let $\kk$ be a field or PID. 
Fix a word $\bw$ as above, and let $A_{\bw,\kk} = \pi_*\kk_{\Xt_\bw}$
be the pushforward to $X$ of the constant sheaf with coefficients in $\kk$.
We say that \emph{the decomposition theorem holds} for
$A_{\bw,\kk}$ if it is isomorphic to a direct sum of
shifted intersection cohomology sheaves $\IC(\overline{X_y};\kk)[s]$.

Let $S^\circ_y = S_y \setminus \{y\}$.  Consider the
natural homomorphism
\[\phi_{y,\bw,\kk}\colon H^\udot(\pi^{-1}(S_y),\pi^{-1}(S^\circ_y);\kk) 
\to H^\udot(\pi^{-1}(S_y);\kk)\]
of cohomology groups.  Because $\pi$ is proper, this is the same as 
the map obtained by applying hypercohomology to the adjunction morphism
$(i_y)_!i_y^!(A_{\bw,\kk}|_{S_y}) \to A_{\bw,\kk}|_{S_y}$, where 
$i_y \colon \{y\} \to S_y$ is the inclusion.  
Applying $i_y^*$ to the adjunction gives a map 
$i_y^!(A_{\bw,\kk}|_{S_y}) \to i_y^*(A_{\bw,\kk}|_{S_y}) = A_{\bw,\kk}|_y$
whose hypercohomology also computes $\phi_{y,\bw,\kk}$,
since there is a cocharacter of $T$ which contracts $S_y$
onto $y$ (see \cite{SprPurity}).

\begin{prop} \label{univ coeff} Both the source and target of the homomorphism $\phi_{y,\bw,\Z}$
are free $\Z$-modules, and they vanish in odd degrees.  Consequently, we have $\phi_{y,\bw,\kk} = \phi_{y,\bw,\Z} \otimes_\Z \kk$.
\end{prop}

\begin{proof}  Using the last remark and the
properness of $\pi$, we see that the target of $\phi_{y,\bw,\Z}$ is
isomorphic to the cohomology of the fiber $\pi^{-1}(y)$.  Since this fiber has
a paving by affines, its cohomology is free and vanishes in odd degrees.  The  
freeness and parity vanishing for the source of  $\phi_{y,\bw,\Z}$ follows from 
the isomorphism
\[
H^k(\pi^{-1}(S_y), \pi^{-1}(S^\circ_y); \kk)  \cong \Hom_\kk(H_{2d-k}(\pi^{-1}(y);\kk),  \kk),
\]
where $d = \dim_\C \pi^{-1}(S_y) = l - \ell(y)$.
This in turn follows from the freeness of $H_{2d-k}(\pi^{-1}(y)); \kk)$, the universal coefficient
theorem, and Poincar\'e duality with supports for the smooth variety $\pi^{-1}(S_y)$ --- 
see \cite[Proposition 3.46]{Hatcher}, for example.  

The last part now follows by the universal coefficient theorem.  
\end{proof}

\begin{prop} \label{obstruction to splitting} If the decomposition theorem holds for $A_{\bw,\kk}$,
then the cokernel of $\phi_{y,\bw,\kk}$ is a free $\kk$-module.
\end{prop}
\begin{proof} 
For any $w\in W$, the map
\[\HH^\udot(i_y^!(\IC(\overline{X_w};\kk)|_{S_y})) \to \HH^\udot(i_y^*\IC(\overline{X_w};\kk)|_{S_y}))\]
is an isomorphism if $y=w$ and is zero if $y \ne w$, by the degree
vanishing for intersection cohomology. As noted above the map
$\HH^\udot(i_y^!(A_{\bw,\kk}|_{S_y})) \to \HH(i_y^*(A_{\bw,\kk}|_{S_y}))$ is
equal to $\phi_{y,\bw,\kk}$, so if the decomposition theorem holds,
this map is a direct sum of maps whose cokernels are free.
\end{proof}

\begin{remark}  Suppose that $\bw$ is a reduced word for $w \in W$, and 
that the resulting map $\pi\colon \Sigma_\bw \to \overline{X_w}$ is 
semi-small.  In this case, the map $\phi_{y,\bw,\kk}$ was previously
considered in \cite{dCM,JMW2}, in the guise of an intersection form 
on the top Borel-Moore homology of $\pi^{-1}(y)$ with $\kk$ coefficients.
Semi-smallness implies that $\phi_{y,\bw,\kk}$ vanishes except in one 
degree, namely $\dim_\C(\overline{X_w}\cap S_y) = \ell(w) - \ell(y)$, 
and if the coefficients $\kk$ are a field or a complete local principal ideal
domain, \cite[Theorem 3.3.3]{dCM} or \cite[Theorem 3.5]{JMW2} shows that $\phi_{y,\bw,\kk}$ 
is an isomorphism for all 
$y$ if and only if the decomposition theorem holds with coefficients in $\kk$.
(Note that the result in \cite{dCM} is stated for $\Q$ coefficients only, but
in fact the arguments work more generally.)

In fact, a straightforward generalization of the argument given in \cite[Section 3]{JMW2}
shows that, even when $\pi$ is not semi-small, the graded 
multiplicity with which the parity sheaf $\sE(y,\kk)$ occurs as a direct summand of 
$A_{\bw, \kk}$ is given by the graded rank of $\phi_{y,\bw,\kk}$.  Hence, if
$\kk$ is a field of characteristic $p$, then the multiplicity with which 
$\sE(y,\kk)$ occurs in $A_{\bw,\kk}$ will be the same as in characteristic
$0$ if and only if $\coker \phi_{y,\bw,\kk}$ has no $p$-torsion.
\end{remark}

It follows from Proposition \ref{obstruction to splitting} that if $\kk$ is a 
field of characteristic $p$ and the cokernel of $\phi_{y,\bw,\Z}$ has $p$-torsion, the decomposition
theorem will fail for $A_{\bw,\kk}$, and so we must have
$\ch \sE(v, \kk) \ne \ch \sE(v, \Q) = \h_v$ for some $v \in W$.
Then Proposition \ref{prop:IC} and Corollary 
\ref{cor:Ztorsion}  imply that the parity sheaf $\sE(v,\kk)$ is not isomorphic to
$\IC_B(\overline{X_v}; \kk)$ and either the 
stalks or costalks of $\IC(\overline{X_v}; \Z)$ will have $p$-torsion.

In general it is difficult to determine for which $v$ this will happen.  It is not
hard to see, however, that at least one such $v$ must lie in the interval
$[y, w]$, where $w$ is the unique maximum element in the set $\{\bw^\vep \mid \vep \in D_l\}$.
The upper bound comes because $\pi(\Sigma_\bw)= \overline{X_w}$, so the sheaf $A_{\bw,\kk}$
is supported on $\overline{X_w}$.  For the lower bound, notice that the argument of
Proposition \ref{obstruction to splitting} still applies if we restrict to the open
set $U_y := \bigcup_{z \ge y} X_y$, so we can conclude that the decomposition theorem
fails for $A_{\bw,\kk}|_{U_y}$.

The following result gives one case where it is possible to be more precise about
the relation between $\phi_{y,\bw,\kk}$ and the stalks and costalks of 
$\IC(\overline{X_w},\kk)$ for a particular $w$.

\begin{prop} \label{isolated} Let $\bw$ be a reduced word for $w \in W$, take $y \le w$, and let
 $V = \overline{X_w} \cap S_y$ and $V^\circ = \overline{X_w} \cap S^\circ_y$.
Suppose that the map $\pi^{-1}(V^\circ) \to V^\circ$ is small.  
Let $\IC = \IC(V; \kk) \cong \IC(\overline{X_w};\kk)|_{V}$, shifted so that
$\IC|_{X_w \cap S_y}$ is a constant local system in degree $0$.  
If $d = \dim_\C V$, 
then the stalks and costalks of $\IC$ at $y$ are given by
\[  \HH^r(i^*_y \IC) = \begin{cases}
\coker \phi^r & \mbox{if }\; r \in 2\Z \;\mbox{ and }\; r < d \\
\ker \phi^{r+1} & \mbox{if }\; r + 1 \in 2\Z \;\mbox{ and }\; r < d\\
0 & \mbox{otherwise}
\end{cases}
\] 
and 
\[  \HH^r(i^!_y \IC) = \begin{cases}
\ker \phi^r & \mbox{if }\; r \in 2\Z \;\mbox{ and }\; r > d \\
\coker \phi^{r-1} & \mbox{if }\; r-1 \in 2\Z \;\mbox{ and }\; r > d \\
0 & \mbox{otherwise}
\end{cases}
\] 
where $\phi = \phi_{y,\bw,\kk}$ and $\phi^r$ is the degree $r$ part of $\phi$.
\end{prop}
\begin{proof}

Put $i = i_y$ and let $j \colon V^\circ \to V$ be the inclusion.
Let $A = A_{\bw, \kk}|_{V}$; it is isomorphic to the pushforward of 
$\underline{\kk}_{\pi^{-1}(V)}$ to $V$. 
Since the map $\pi^{-1}(V^\circ) \to V^\circ$ is
small, we have $j^*A \cong \IC(V^\circ;\kk)$.

From the truncation triangle
\[\IC = \tau_{\le d-1}j_*j^*A \to j_*j^*A \to \tau_{\ge d}j_* j^*A \stackrel{[1]}\longrightarrow\]
it follows that $\HH^r(i^*\IC) = \HH^r(i^*j_*j^*A)$ if $r < d$,
and $\HH^r(i^!\IC) = \HH^{r-1}(i^*j_*j^*A)$ if $r > d$, and they vanish otherwise.
(For the second statement, use the fact that $\tau_{\ge d}j_* j^*A$ is supported at $y$,
so $i^!\tau_{\ge d}j_* j^*A \cong i^*\tau_{\ge d}j_* j^*A$.)

As noted before Proposition \ref{univ coeff}, the map 
$\phi$ is the hypercohomology of the natural map $i^!A \to i^*A$, so
applying hypercohomology to the triangle
\[ i^! A \to i^*A \to i^*j_*j^*A \stackrel{[1]}\longrightarrow\]
and using Proposition \ref{univ coeff} (parity vanishing) proves the proposition.
\end{proof}

\begin{remark} \label{smooth Euler}
If the fiber $Y := \pi^{-1}(y)$ is smooth, then by passing to 
a tubular neighborhood and using excision we can replace the pair
$(\pi^{-1}(S_y), \pi^{-1}(S^\circ_y))$ by $(N, N^\circ)$, where
$N$ is the total space of the normal bundle $\cN$ to $Y$ in 
$\pi^{-1}(S_y)$, and $N^\circ$ is the complement of the the zero section in $N$.
In this case, by the Thom isomorphism theorem we have $H^\udot(N,N^\circ;\kk) \cong
H^{\udot - 2d}(Y; \kk)$, $d = \rank \cN$, and the map 
$\phi_{y,\bw,\kk}$ can be identified with 
multiplication by the Euler class $e(\cN)$.
\end{remark}




\subsection*{The hexagon permutation}
Now fix $G = GL(8,\C)$, so $W$ is the symmetric group on the set
$\{1,\dots, 8\}$.  Taking the torus $T$ to be the diagonal matrices,
the lattice $X(T)$ of characters is naturally identified with 
$\Z^n$.  Let $\be_i$ be the $i$th standard basis vector of $\Z^n$.
The roots of $G$ are then the vectors $\rho_{ij} = \be_i-\be_j$, 
for $1\le i,j\le 8$, $i\ne j$.  

Choose the Borel subgroup $B$ to be the lower triangular matrices.
With this choice, the positive roots are $\rho_{ij}$ with $i < j$, and
the simple roots are $\rho_i \defeq \rho_{i,i+1}$, $i=1,\dots,7$.
The simple reflection $s_{\rho_i}$ corresponding to $\rho_i$
is the transposition of $i$ and $i+1$.  





\smallskip
We now fix $w = w_1$ where $w_1$ is the shortest of the ``hexagon
permutations'' introduced in \ref{subsec:A7}. The one-line notation of
$w$ is $46718235$.
A reduced word $\bw$ for $w$ can be given by the sequence of simple reflections 
corresponding to the sequence of simple roots
\[\bal = (\alpha_1,\dots,\alpha_{14}) = (
\rho_3,\rho_2,\rho_1,\rho_5,\rho_4,\rho_3,\rho_2, \rho_6,\rho_5,\rho_4,
\rho_3,\rho_7,\rho_6,\rho_5).\]
Reduced words for the other three hexagon permutations are obtained
from this by appending $\rho_4$ at the beginning or the end, or both.

Let $y = s_{\rho^{}_2}s_{\rho^{}_3}s_{\rho^{}_2}s_{\rho^{}_5}s_{\rho^{}_6}s_{\rho^{}_5}$.
It is given in one-line notation as $14327658$.  

\begin{thm} \label{hexagon pairing}
The source and target of the map
\[\phi^8=\phi^8_{y,\bw,\Z}\colon H^8(\pi^{-1}(S_y), \pi^{-1}(S^\circ_y)) \to H^8(\pi^{-1}(S_y))\]
(the degree eight part of $\phi_{y,\bw,\Z}$) 
are both isomorphic to $\Z$; its cokernel is isomorphic to $\Z/2\Z$.
\end{thm}
If $2$ is not a unit in $\kk$, 
it follows from this and Proposition \ref{obstruction to splitting} that
the decomposition theorem fails with $\kk$ coefficients.  
Furthermore, one can check that 
the map $\pi\colon \Sigma_{\bw} \to \overline{X_w}$ is a small resolution 
over the open set $\bigcup_{x > y}X_x$, so Proposition \ref{isolated} implies 
that $\IC(\overline{X_w}; \kk)$ has $2$-torsion in its costalk at $y$
if $\kk = \Z$, and has nonvanishing stalks and costalks in odd degrees
if $\mathop{\mathrm{char}} \kk = 2$. 

To prove Theorem \ref{hexagon pairing}, we look more closely at the fiber $Y= \pi^{-1}(y)$.
We compute the map $\phi^8$ using $T$-equivariant
cohomology and localization.  First we describe the $T$-fixed 
points in $Y$.

Recall that we denote the standard basis of $D_l$ by
$\{\delta_i\}$.  Define elements of $D_{14}$ by
\[\begin{matrix}
    \lambda_1 = \delta_1 + \delta_2 +\delta_6 &
  \mu_1 = \delta_4 + \delta_8 + \delta_9 \\
    \lambda_2 = \delta_2 + \delta_6 + \delta_7 &
  \mu_2 = \delta_8 + \delta_9 + \delta_{13} \\
    \lambda_3 = \delta_6 + \delta_7 + \delta_{11} &
  \mu_3 = \delta_9 + \delta_{13} + \delta_{14} \\
    \lambda_4 = \delta_7 + \delta_{11} + \delta_1 &
  \mu_4 = \delta_{13} + \delta_{14} + \delta_4 \\
    \lambda_5 = \delta_{11} + \delta_1 + \delta_2 &
  \mu_5 = \delta_{14} + \delta_4 + \delta_8
\end{matrix}\]
and 
\[\nu = \delta_5 + \delta_{10}.\]

\begin{prop}\label{Y components} There are $29$ $T$-fixed points in $Y = \pi^{-1}(y)$.
They are given by
\begin{itemize}
\item[(a)] $p(\lambda_i + \mu_j)$, $1 \le i,j \le 5$, and
\item[(b)] $p(\lambda_i + \mu_j + \nu)$, $i, j \in \{4,5\}$.  
\end{itemize}
$Y$ has two irreducible components.  The
first, call it $Y_1$, is isomorphic to $Z \times Z$, where
$Z$ is isomorphic to $\PP^2$ blown
up at two points.  The $T$-fixed points in $Y_1$ are the ones given by 
(a) above.

The other component $Y_2$ is isomorphic to 
$\PP^1 \times \PP^1 \times \PP^1$. Its $T$-fixed points are
the four of type (b) and the four of type (a) where $i,j \in \{4,5\}$.
\end{prop}

\begin{remark}  We do not need the full statement of the proposition to prove Theorem \ref{hexagon pairing};
we only need the description of the fixed points and the component $Y_1$ and the fact that all other
components have smaller dimension.
\end{remark}

\begin{proof} The enumeration of the points of $Y^T$ is straightforward, using Proposition \ref{BS fixed points}.

Let $I = \{1, 2, 4, 6, 7, 8, 9, 11, 13, 14\}$.  Then $w_I$
is the (non-reduced) subword of $\bw$ corresponding to the sequence of simple roots
\[
(\rho_3,\rho_2,\rho_5,\rho_3,\rho_2, \rho_6,\rho_5,
\rho_3,\rho_6,\rho_5),\] 
where the roots $\rho_1$, $\rho_4$ and $\rho_7$
have been omitted from $\bw$.  As before we identify $\Xt_{\bw_I}$ with a subvariety of $\Xt_\bw$.

The simple reflections $\rho_2, \rho_3, \rho_5,\rho_6$ that appear 
in $\bw_I$ generate the Weyl group of the 
group $GL(3) \times GL(3)$, embedded into $GL(8)$ as block diagonal matrices 
acting on the middle two factors in the decomposition $\C^8 = \C \oplus \C^3 \oplus \C^3 \oplus
\C$.  It follows that there is an isomorphism $\Xt_{\bw_I} \cong \Xt_1 \times \Xt_2$, where
the factors are the Bott-Samelson varieties for $(\rho_3,\rho_2,\rho_3,\rho_2,\rho_3)$
and $(\rho_5,\rho_6,\rho_5,\rho_6,\rho_5)$, respectively.  Both $\Xt_1$ and $\Xt_2$ are
isomorphic to the Bott-Samelson variety in Example \ref{SL3 example}, and it is easy to
see that $Y_1 := Y\cap \Xt_{\bw_I}$ is a product of two copies of the fiber from that example.

To see that $Y_1$ is an irreducible component of $Y$, note that 
a computation with Theorem \ref{BB cells} shows that the paving by affines 
of $Y$ given by intersecting with the Bia\l ynicki-Birula cells has only
one cell of dimension four, namely $Y \cap \Sigma_{\bw,\lambda_1+\mu_1}$, and
all other cells are of smaller dimension.  So the closure of this
cell must be a component of $Y$, and it is the only four-dimensional component, so
it is equal to $Y_1$.  

To understand the other component, note that there is only one cell in
$Y \setminus Y_1$ of dimension three, namely $\Sigma_{\bw,\lambda_5 +\mu_5 + \nu} \cap Y$, and all other cells are of 
smaller dimension.  Consider the subword $\bw_J$, where
\[J = \{1, 2, 4, 5, 7, 8, 10, 11, 13, 14\}.\]
It corresponds to the sequence of simple roots
\[(\rho_3, \rho_2, \rho_5, \rho_4, \rho_2, \rho_6, \rho_4, \rho_3, \rho_6, \rho_5).\] 
This is the smallest subword of $\bw$ 
containing all the nonzero entries of $\lambda_4, \lambda_5, \mu_4, \mu_5$, and $\nu$.  
Set $Y_2 = Y \cap \Xt_{\bw_J}$.    The $T$-fixed
points of $Y_2$ are the ones of type (b) and the four of
type (a) with $i, j \in \{4, 5\}$.

To see that $Y_2$ is isomorphic to $(\PP^1)^3$, we use the following facts, which are 
easily checked:
\begin{itemize} 
\item If words $\bw_1$ and $\bw_2$ differ by interchanging adjacent 
transpositions $s_{\rho_i}$ and $s_{\rho_j}$ with $|i - j| > 1$, then the 
Bott-Samelson varieties $\Sigma_{\bw_1}$ and $\Sigma_{\bw_2}$ are isomorphic
by a $T$-equivariant map which commutes with the projections $\pi_1, \pi_2$ to $X$.
\item If the word $\bw_1$ is obtained from $\bw_2$ by doubling 
$k$ of the simple reflections (i.e.\ replacing $s_{\rho_i}$ with $s_{\rho_i} s_{\rho_i}$), 
then $\Xt_{\bw_1}$ is a fiber bundle over $\Xt_{\bw_2}$ with fiber $(\PP^1)^k$, so that the composition
$\Xt_{\bw_1} \to \Xt_{\bw_2} \stackrel{\pi_1}\longrightarrow X$ gives the map $\pi_2$.
\end{itemize}
Using these, if we let $\bw_1$ and $\bw_2$ correspond to the sequences
\[(\rho_3, \rho_2,\rho_2, \rho_5, \rho_4, \rho_4, \rho_3, \rho_6, \rho_6, \rho_5)\;\text{and}\;
 (\rho_3, \rho_2, \rho_5, \rho_4, \rho_3, \rho_6, \rho_5)\]
 of simple roots,
 then we see that $\Xt_{\bw_J} \cong \Xt_{\bw_1}$ is a fiber bundle with 
 fiber $\PP^3$ over $\Xt_{\bw_2}$.  It is easy to see that the fiber 
 of $\Xt_{\bw_2} \to X$ over $y$ is a single point, so $Y_2$, which is the 
 fiber of $\Xt_{\bw_J} \to X$ over $y$, is isomorphic to $(\PP^1)^3$.

Using Corollary 
\ref{Containment of cells} it is easy to check that all the cells $Y \cap \Sigma_{\bw, \vep}$ which are
not contained in $Y_1$ are contained in $Y_2 := Y \cap \Xt_{\bw_J}$.
\end{proof}

The first part of Theorem \ref{hexagon pairing} follows immediately.  Although the fiber 
$Y$ is not smooth, there is only one component of dimension four, so
the target of $\phi^8$ is \[H^8(\pi^{-1}(S_y)) \cong H^8(Y) \cong H^8(Y_1) \cong \Z.\]
Dually, we have isomorphisms 

\begin{align*} H^8(\pi^{-1}(S_y), \pi^{-1}(S^\circ_y)) & \cong H^8(\pi^{-1}(S_y), \pi^{-1}(S_y) \setminus Y) \\
& \cong H^8(\pi^{-1}( S_y), \pi^{-1}(S_y) \setminus Y_1) \cong \Z.
\end{align*}

Thus we can reduce the computation of $\phi^8$ to the smooth case: it is isomorphic to 
the restriction map $H^8(N, N\setminus Y_1)\to H^8(N) \cong H^8(Y_1)$, where $N$ is
the total space of the normal bundle $\cN$ to $Y_1$ in $\pi^{-1}(S_y)$.  As remarked earlier,
this can be identified with 
the map $H^0(Y_1) \to H^8(Y_1)$ given by multiplication by the Euler class $e(\cN)$,
so the image of $\phi^8$ is spanned by $e(\cN)$. 

We will compute this class by computing the equivariant Euler class $e_T(\cN)\in H^8_T(Y_1)$ 
and then finding its image in ordinary cohomology.
To do this, we split the normal bundle $\cN$ into line bundles.  
We have seen in the
proof of Proposition \ref{Y components} that $Y_1 = \Xt_{\bw_I} \cap \pi^{-1}(S_y)$.
By Lemma \ref{transverse} this intersection is transverse, and so $\cN$ is 
isomorphic to the restriction to $Y_1$ of the 
normal bundle to $\Xt_{\bw_I}$ in $\Xt_\bw$.

Let $I_1 = I \cup \{3\}$, $I_2 = I \cup \{5\}$, $I_3 = I \cup \{10\}$, $I_4 = I \cup \{12\}$,
so $\bw_1,\dots \bw_4 := \bw_{I_1}, \dots, \bw_{I_4}$ are all the 
subwords of $\bw$ of length $11$ which contain $\bw_I$ as a subword.  
It is easy to see that the subvarieties $\Xt_{\bw_i}$ intersect transversely 
in $\Xt_\bw$, so letting $\cL_i$ be 
the restriction to $Y_1$ of the normal bundle to $\Xt_{\bw_I}$ in $\Xt_{\bw_{i}}$, 
we have a splitting $\cN \cong \cL_1 \oplus \cL_2 \oplus \cL_3 \oplus \cL_4$. 

To compute the classes $e_T(\cL_i)$, we compute their restrictions
to the fixed point set $Y_1^T = \{p(\lambda_j + \mu_k) \mid 1 \le j, k \le 5\}$.
The restriction $e_T(\cL_i)|_{p(\lambda_j + \mu_k)} \in H_T^2(p(\lambda_j + \mu_k)) \cong X(T)$
is just the $T$-weight of the tangent space to the unique $T$-curve containing
$p(\lambda_j + \mu_k)$, contained in $\Xt_{\bw_i}$ and not contained in $\Xt_{\bw}$.
This curve is the curve joining $p(\lambda_j + \mu_k)$ and $p(\eta_i+ \lambda_j + \mu_k)$,
where $\eta_1 = \delta_3$, $\eta_2 = \delta_5$, $\eta_3 = \delta_{10}$, $\eta_4 = \delta_{12}$.
Using Propositions \ref{flag 1d orbits} and \ref{BS fixed points}, we can compute that 
its $T$-weight is the sum of the entries under $\lambda_j$ and $\mu_k$ in the following tables:

\[\begin{array}{|c|c|c|c|c|c|}
\hline  & \lambda_1 & \lambda_2 & \lambda_3 & \lambda_4 & \lambda_5  \\ 
\hline\hline e_T(\cL_1) & \rho_1+\rho_2+\rho_3 & \rho_1+\rho_2 & \rho_1 & \rho_1 & \rho_1+\rho_2+\rho_3  \\ 
\hline e_T(\cL_2) & \rho_3+\rho_4 & \rho_4 & \rho_4 & \rho_3+\rho_4 & \rho_3+\rho_4  \\ 
\hline e_T(\cL_3) & \rho_2+\rho_3+\rho_4  & \rho_2+\rho_3+\rho_4 & \rho_3+\rho_4 &  \rho_3+\rho_4 & \rho_3+\rho_4 \\
\hline e_T(\cL_4) & 0 & 0 & 0 & 0 & 0 \\ 
\hline 
\end{array} 
\]

\[\begin{array}{|c|c|c|c|c|c|}
\hline  & \mu_1 & \mu_2 & \;\;\;\mu_3\;\;\; & \;\;\;\mu_4\;\;\; & \mu_5 \\ 
\hline\hline e_T(\cL_1)  & 0 & 0 & 0 & 0 & 0 \\ 
\hline e_T(\cL_2) & \rho_5 & 0 & 0 & \rho_5 & \rho_5 \\ 
\hline e_T(\cL_3) & \rho_5+\rho_6 & \rho_5+\rho_6 & \rho_5 & \rho_5 & \rho_5  \\
\hline e_T(\cL_4) & \rho_5+\rho_6+\rho_7 & \rho_6+\rho_7 & \rho_7 & \rho_7  & \rho_5+\rho_6+\rho_7 \\ 
\hline 
\end{array}
\]

\medskip

The equivariant class $e_T(\cN) = e_T(\cL_1)e_T(\cL_2)e_T(\cL_3)e_T(\cL_4)$ induces the same class in 
$H^8(Y_1)$ as
\[(e_T(\cL_1) - (\rho_1+\rho_2))(e_T(\cL_2) - (\rho_3 + \rho_4 + \rho_5))(e_T(\cL_3) - (\rho_3 + \rho_4 + \rho_5))(e_T(\cL_4) - (\rho_6 + \rho_7)),\]
where we abuse notation and write a weight $\rho_i \in X(T) = H^2_T(pt)$ instead of its pullback 
under the map $Y_1 \to pt$.  After a little computation one sees that this class restricts to zero at every point 
of $(Y_1)^T$ except $p(\lambda_3+\mu_1)$ and $p(\lambda_1+\mu_3)$, where it has the same restriction as
$e_T(\cT_{Y_1})$, the equivariant Euler class of the tangent bundle to $Y_1$.  (To compute the localization
of $e_T(\cT_{Y_1})$ to the fixed points, use the identification of the weights of $T$-curves in
Example \ref{SL3 example}.  Note that the labeling of the fixed points $\vep_1, \dots, \vep_5$ in that example
corresponds to the labeling of the fixed points $\lambda_1,\dots,\lambda_5$ and $\mu_1,\dots,\mu_5$
of $\Xt_1$ and $\Xt_2$.)  Then the Atiyah-Bott-Berline-Vergne localization formula \cite{AtBo, BeVe} 
implies that $e(\cN)$ is twice a generator of $H^8(Y_1)$, completing the proof of 
Theorem \ref{hexagon pairing}.

\begin{remark} \label{rem:otherhex}
The other hexagon permutations can be shown to have $2$-torsion by a similar computation;
we give only the main points.  Let $\tilde\bw = (s_4)^a \bw (s_4)^b$ and
$\tilde y = (s_4)^a y (s_4)^b$ for $a, b \in \{0,1\}$.  
The fiber $\tilde Y = \pi^{-1}(\tilde y)$ is still four-dimensional, 
but now the union of the components of maximal dimension is isomorphic to $Z_a \times Z_b$, 
where $Z_0 = Z$ and $Z_1 = Z \cup (\PP^1 \times \PP^1)$, the union taken
so that $\{0\} \times \PP^1$ is identified with a $T$-curve in 
$Z$ with trivial normal bundle.  The excess intersection formula \cite{FulInt}
then implies that the matrix of $\phi^8$ is diagonal under the
natural bases given by the components of $\tilde Y$ (in other words,
the components are orthogonal under the intersection form).
The normal bundle to the component $Z \times Z$ is the same as
before, so we have $\det \phi^8 \in 2\Z$.  
\end{remark}

\subsection*{Torsion example in D4}
Let $G = SO(8;\C)$.  We follow the notation of Section \ref{subsec:D4}: the simple reflections in $W$ are $s, t, u, v$ where
$s, u, v$ all commute with each other.  Let $\bw$ be the word $(s,u,v,t,s, u, v)$, put $w = \pi(\bw)$, 
and let $y = suv$.  

\begin{prop} The $T$-fixed points in $Y := \pi^{-1}(y)$ are  
\[\{p(\vep) \in D_7 \mid \vep(4) = 0\;\text{and}\; \vep(i) + \vep(i+4) = 1\;\text{for}\; i = 1,\dots, 3\}.\]
The fiber $Y$ is the transverse intersection of $\Xt_{(s,u,v,s,u,v)} \subset \Xt_\bw$ and $\pi^{-1}(S_y)$.
It is $T$-equivariantly isomorphic to $\PP^1 \times \PP^1 \times \PP^1$, where the $T$-weights
on the three factors are $\rho_s$, $\rho_u$, and $\rho_v$, respectively.  
\end{prop}

By Remark \ref{smooth Euler}, since $Y$ is smooth, we have $H^\udot(\pi^{-1}(S_y),\pi^{-1}(S^\circ_y);\kk) \cong H^{\udot-2}(Y)$, and 
by Remark \ref{smooth Euler} the map $\phi = \phi_{y,\bw,\Z}$ can be identified with multiplication by 
$e(\cL)$ on $H^\udot(Y)$, where $\cL$ is the normal bundle to $Y$ in $\pi^{-1}(S_y)$.  
As in the previous example we compute this by computing the localization of the equivariant
class $e_T(\cL)$ to the fixed points $Y^T$.  We have
\[e_T(\cL)|_{p(\vep)} = \rho_t + \vep(1)\rho_s + \vep(2)\rho_u + \vep(3)\rho_v,\]
so $e(\cL) = \alpha + \beta+\gamma$, where $\alpha,\beta,\gamma \in H^2(Y)$ are the
pullbacks of a generating class of $H^2(\PP^1)$ by the three projection maps.
Multiplication by this class from $H^2(Y)$ to $H^4(Y)$ is given by the matrix
\[\left[\begin{matrix} 1 & 1& 0 \\
1 & 0 & 1\\
0 & 1 & 1\\
\end{matrix}\right]
\]  
with respect to the natural monomial basis in $\alpha$, $\beta$, $\gamma$.  This 
matrix has determinant $-2$, so $\coker \phi^4$ has $2$-torsion.  

Just as we
saw for the hexagon permutation, if $2$ is not a unit in $\kk$
Proposition \ref{obstruction to splitting} implies that the 
decomposition theorem with $\kk$ coefficients fails 
and Proposition \ref{isolated} implies that the costalk 
of $\IC(\overline{X_w}; \kk)$ at $y$ does not vanish in odd degrees.  
Furthermore, if $\mathop{\mathrm{char}} \kk = 2$, the stalk of
$\IC(\overline{X_w}; \kk)$ also has nonvanishing odd-degree part.

\def\cprime{$'$} \def\cprime{$'$}


\end{document}